\documentclass[reqno]{amsart}

\usepackage[alphabetic]{amsrefs}

\usepackage{amssymb}

\usepackage[all,cmtip]{xy}

\PassOptionsToPackage{usenames,dvipsnames}{xcolor}
\usepackage[inline]{enumitem}
\usepackage{mathtools}
\usepackage{xfrac}
\usepackage{bm}
\usepackage{tkz-euclide}
\usepackage{tikz,xifthen}
\usepackage{tikz-cd}
\usepackage{mathrsfs}
\usetikzlibrary{arrows}
\usetikzlibrary{decorations.markings}
\usetikzlibrary{backgrounds}
\usetikzlibrary{patterns}
\usetikzlibrary{tikzmark}
\usetikzlibrary{calligraphy}
\usepackage{pgfplots}
\pgfplotsset{compat=1.10}
\usepgfplotslibrary{fillbetween}
\usepackage[T1]{fontenc}
\usepackage{tensor}
\usepackage{lmodern}
\usepackage{cancel}
\usepackage{multicol}
\usepackage{hyperref}
\usepackage{adjustbox}
\usepackage{graphicx}
\usepackage{yhmath}
\usepackage{etoolbox}
\usepackage{stmaryrd}
\usepackage{needspace}
\usepackage{listofitems}
\usepackage{geometry}
\usepackage[dvipsnames]{xcolor}
\setlist{leftmargin=8mm, itemsep=0.4em, topsep={0.5em}}
\usepackage[textsize=tiny,linecolor=MidnightBlue,backgroundcolor=MidnightBlue!25,textwidth=2.7cm]{todonotes}
\makeatletter 
\@mparswitchfalse%
\makeatother
\normalmarginpar

\hypersetup{
    colorlinks=true,
    linkcolor=MidnightBlue!90!black,
    citecolor=black!25!red,
    urlcolor=MidnightBlue!90!black,
    pdftitle={Yang-Mills theory for Multiplicative Ehresmann connections},
    pdfauthor={Žan Grad},
    pdfpagemode=FullScreen,
}

\makeatletter
\renewcommand\subsection{
\Needspace{2cm}
\@startsection{subsection}{2}%
  \z@{-.5\linespacing\@plus-.7\linespacing}{.5\linespacing}%
  {\bf}}
\makeatother

\makeatletter
\newsavebox{\@brx}
\newcommand{\llangle}[1][]{\savebox{\@brx}{\(\m@th{#1\langle}\)}%
  \mathopen{\copy\@brx\kern-0.5\wd\@brx\usebox{\@brx}}}
\newcommand{\rrangle}[1][]{\savebox{\@brx}{\(\m@th{#1\rangle}\)}%
  \mathclose{\copy\@brx\kern-0.5\wd\@brx\usebox{\@brx}}}
\makeatother

\makeatletter
\def\blfootnote{\xdef\@thefnmark{}\@footnotetext}
\makeatother
\setlist[enumerate]{leftmargin=25pt, label={(\roman*)}}

\DeclarePairedDelimiterX\set[1]\lbrace\rbrace{\def\given{\;\delimsize\vert\;}#1}
\renewcommand{\d}[1]{\ensuremath{\operatorname{d}\!{#1}}}
\newcommand{\D}[1]{\ensuremath{\operatorname{D}\!{#1}}}

\newcommand{\enquote}[1]{``{#1}''}
\newcommand{\rra}{\rightrightarrows}
\newcommand{\Ra}{\Rightarrow}
\newcommand{\ra}{\rightarrow}

\newcommand{\R}{\mathbb{R}}

\newcommand{\Z}{\mathbb{Z}}

\newcommand{\C}{\mathcal{C}}
\newcommand{\G}{\mathcal{G}}
\newcommand{\DD}{\mathscr{D}}

\newcommand{\vb}{\mathcal{VB}}
\newcommand{\A}{\mathscr{A}}
\newcommand{\F}{\mathcal{F}}
\renewcommand{\H}{\mathcal{H}}
\newcommand{\ve}{\mathscr{V}\hspace{-0.25em}\mathscr{E}}

\newcommand{\ul}[1]{\underline{\smash{#1}}}

\newcommand{\PB}{\mathcal{PB}}
\renewcommand{\O}{\mathcal{O}}
\renewcommand{\S}{\mathscr{S}}
\renewcommand{\sec}{§}
\renewcommand{\L}{\mathscr{L}}
\renewcommand{\frak}{\mathfrak}

\newcommand{\End}{\operatorname{End}}

\newcommand{\codim}{\operatorname{codim}}
\newcommand{\sgn}{\operatorname{sgn}}

\newcommand{\pr}{\mathrm{pr}}
\newcommand{\vol}{\mathrm{vol}}
\newcommand{\Ad}{\mathrm{Ad}}

\newcommand{\ad}{\operatorname{ad}}
\newcommand{\Hor}{\mathrm{Hor}}
\newcommand{\id}{\mathrm{id}}

\newcommand{\im}{\operatorname{im}}
\newcommand{\inv}{\mathrm{inv}}
\newcommand{\ginv}{{\G\text{-}\inv}}
\newcommand{\ainv}{{A\text{-}\inv}}

\newcommand{\vf}{\mathfrak X}

\newcommand{\deriv}[2]{\frac{d}{d{#1}}\Big|_{{#1}={#2}}}

\newcommand{\arc}{\mathscr{C}}
\newcommand{\cev}[1]{{#1}^L}
\newcommand{\inner}[2]{\left\langle{#1},{#2}\right\rangle}
\newcommand{\innersmall}[2]{\langle{#1},{#2}\rangle}
\newcommand{\innerr}[2]{\llangle #1, #2\rrangle}

\DeclareRobustCommand{\gobblefive}[5]{}
\newcommand*{\SkipTocEntry}{\addtocontents{toc}{\gobblefive}}

\numberwithin{equation}{section}

\theoremstyle{plain} 
\newtheorem{theorem}{Theorem}
\newtheorem{corollary}[theorem]{Corollary}
\newtheorem{lemma}[theorem]{Lemma}
\newtheorem{proposition}[theorem]{Proposition}

\theoremstyle{definition}
\newtheorem{definition}[theorem]{Definition}

\theoremstyle{remark}
\newtheorem{remark}[theorem]{Remark}

\newtheorem*{remark*}{Remark}
\newtheorem{example}[theorem]{Example}

\AtBeginEnvironment{lemma}{\Needspace{1cm}}
\AtBeginEnvironment{theorem}{\Needspace{1cm}}
\AtBeginEnvironment{proposition}{\Needspace{1cm}}
\AtBeginEnvironment{definition}{\Needspace{1cm}}
\AtBeginEnvironment{corollary}{\Needspace{1cm}}

\calclayout

\newcommand{\Addresses}{{
  \bigskip
  \footnotesize

  Žan Grad, \textsc{KU Leuven, Department of Mathematics, Celestijnenlaan 200B, Leuven, Belgium.}\par\nopagebreak

  \vspace{0.2em}
  \textit{E-mail address}: \url{grad.zan@kuleuven.be}.\par\nopagebreak
  
  \textit{Webpage}: \url{https://zangrad.github.io}.
}}

\begin{document}

\title[Yang--Mills theory for multiplicative Ehresmann connections]{Yang--Mills theory for\\multiplicative Ehresmann connections
}

\author[Ž.\ Grad]{Žan Grad} \thanks{This research is supported by PhD Grant UI/BD/152069/2021 of FCT, Portugal, and  COST Action CaLISTA (CA21109) of the European Cooperation in Science and Technology. It was partially supported by FWO and FNRS under FWO project G0B3523N and EOS project G0I2222N}





\subjclass[2020]{58H05 (22A22, 58E15, 53B15)}

\begin{abstract}
  We develop a twofold generalization of classical Yang--Mills theory, extending it from principal bundles to the setting of possibly \textit{non-transitive} and \textit{non-integrable} Lie algebroids. The classical theory is recovered when one considers the Atiyah algebroid of a principal bundle. In our framework, principal bundle connections are replaced by the more general notion of (infinitesimal) multiplicative Ehresmann connections. 
   An action functional for such connections is constructed, now including a curvature 3-form contribution, alongside the usual curvature 2-form term, and the resulting variational problem is naturally constrained by a cohomological condition. 
  We derive the associated Euler--Lagrange equations, and define a class of self-dual solutions (instantons) in both 4 and 5 dimensions. We also show that the solution space is invariant under gauge transformations, and compute its tangent space at a solution.
  As an important example, we show that our framework produces a Yang--Mills theory for connections on bundle gerbes.
\end{abstract}


\maketitle
\thispagestyle{empty}


\tableofcontents

\vspace{-2em}


\pagebreak
\section{Introduction and main results}
\label{sec:intro}
In the framework of principal bundles, Yang--Mills theory is an application of variational calculus to principal bundle connections. Roughly speaking, the \textit{Yang--Mills action functional} inputs principal connections $\omega$ and outputs the $L^2$-norm of the curvature $F^\omega$,
\begin{align}
   \label{eq:classical_YM_action}
   \S(\omega)= \int_M \inner{F^\omega}{F^\omega} \vol_M.
\end{align}
Minimizing this functional amounts to the second-best scenario to having a flat connection. Finding the minima requires solving the corresponding variational problem, i.e., finding its critical points, which are characterized as solutions to the Euler--Lagrange equations. Specifically, for the action functional above, these equations boil down to the vanishing exterior covariant coderivative of $F^\omega$, 
\begin{align}
   \label{eq:classical_YM_eqn}
   \d{}^{\nabla^\omega}{\star}\ F^\omega=0.
\end{align}
This is called the \textit{Yang--Mills equation}; in coordinates, it is a system of nonlinear second-order partial differential equations for $\omega$. One of its main features is that the solution space is preserved by pullbacks along gauge transformations, allowing for the study of the moduli space of solutions. 
Another important feature is that when the base space $M$ is 4-dimensional, the Yang--Mills equations allow for a distinguished first-order reduction given by the (anti) self-duality equations $F^\omega = \pm \star F^\omega$, whose solutions, known as \textit{instantons}, are  global minima of the action. This paper is about a natural generalization of this framework and its features. For a reference on the classical Yang--Mills theory (on principal bundles), we direct the reader to \cites{atiyah_geometry_of_ym_fields} or \cite{hamilton}.

A concrete motivation for our generalization comes from the theory of \textit{bundle gerbes} over manifolds. These are geometric models for higher degree cohomology classes. To briefly elaborate, recall that for a smooth manifold $M$, there is a one-to-one correspondence between second degree integral cohomology classes $H^2(M;\Z)$ and (isomorphism classes of) principal $S^1$-bundles over $M$, via their Chern classes. Similarly, the third degree integral cohomology classes $H^3(M;\Z)$ correspond to (Morita equivalence classes of) \textit{$S^1$-bundle gerbes} over $M$ \cites{bundle_gerbes_original, bundle_gerbes, stacks_gerbes}. Just as principal bundles, they also admit a natural notion of a connection, known in string theory under the name of \textit{Kalb--Ramond fields}. 
 A natural question, posed in \cite{ym_gerbes}, is whether Yang--Mills theory can be extended to bundle gerbes. We provide a positive answer to this question in Theorem \ref{thm:ym_gerbes}, as a direct consequence of a more general approach with \textit{multiplicative Ehresmann connections} (abbr.\ MECs). 
 
 The notion of a MEC was first introduced in \cite{gerbes} for Lie groupoid extensions, and later systematically studied in \cite{mec} in terms of their existence, properties, and examples. They are a certain kind of connections compatible with the geometric structure of a given Lie algebroid or groupoid. Beyond the familiar case of principal bundle connections and linear connections on vector bundles, examples include connections on (non-abelian) bundle gerbes, equivariant splittings for bundles of ideals on action groupoids, and more. They are interesting in their own right and exhibit a strong applicative character. For instance, they have recently been used to obtain a normal form theorem for Poisson submanifolds \cite{poisson_submanifolds}, and to study fibrations of Lie groupoids \cite{mecs_for_fibrations}.  Moreover, in \cite{mec}, MECs were also introduced in the infinitesimal setting of Lie algebroids. This is particularly important, since, as we will see, our generalization of Yang--Mills theory can be built without the integrability assumption. To emphasize they belong to the infinitesimal realm, we will often call them \textit{IM connections}, (or \textit{infinitesimal} MECs, or IMECs). 

\SkipTocEntry\subsection*{Overview of main results}
In this paper, we generalize Yang--Mills theory 
in a two-fold way, relaxing both the \textit{transitivity} and the \textit{integrability} conditions that were implicitly assumed when dealing with principal bundles. We obtain a framework suitable for general Lie groupoids and Lie algebroids; connections on principal bundles are hereby replaced with the more general MECs on Lie groupoids and IMECs on Lie algebroids. Our main results are the following.

\begin{itemize}
   \item We construct an action functional for (infinitesimal) MECs, generalizing \eqref{eq:classical_YM_action}. We pose a natural \textit{constrained variational problem} and obtain the Euler--Lagrange equations, which boil down to a PDE similar to \eqref{eq:classical_YM_eqn}, as well as a certain partial differential relation (PDR),
   \item \textit{Gauge invariance.} The solution space of the variational problem is proved to be \textit{invariant under Lie groupoid (or algebroid) automorphisms} preserving the extra data entering the construction. 
   \item \textit{Instantons in 4D and 5D.} We extend the classical notion of self-dual fields on a four-dimensional base, introduce a five-dimensional analogue, and show they form a special class of solutions,
   \item We relate the obtained Yang--Mills theories in the global and the infinitesimal setting, via the \textit{van Est map} --- a fundamental map relating the complexes of differential forms central to our framework. We show that if a given algebroid is integrable, the global and infinitesimal Yang--Mills theory are equivalent, up to connectedness conditions on the integrating Lie groupoid.
   \item The (formal) tangent space to the space of solutions is obtained,
   \item The theory is presented on several classes of examples: (i) bundle gerbes (non-transitive, integrable example); (ii) the Almeida--Molino algebroid (transitive, generally non-integrable example); and (iii) a totally intransitive example, admitting self-dual solutions in 5 dimensions.
\end{itemize}
The rest of the introduction is dedicated to an exposition of the first three bullet points.

\SkipTocEntry\subsection*{The classical theory from a Lie algebroid perspective}
To expand on our stated results and motivate them, we first consider principal bundles from the viewpoint of Lie algebroids. In the work of Atiyah and Bott \cite{atiyah-bott}, it is already recognized that connections on a principal $G$-bundle $\pi\colon P\ra M$ can equivalently be described as vector bundle splittings of the following short exact sequence, called the \textit{Atiyah sequence}.
\begin{align}
   \label{eq:atiyah_ses}
   \begin{tikzcd}[ampersand replacement=\&]
   	0 \& {\ad (P)} \& {A(P)} \& TM \& 0.
   	\arrow[from=1-1, to=1-2]
   	\arrow[from=1-2, to=1-3]
   	\arrow["\rho"', from=1-3, to=1-4]
   	\arrow[from=1-4, to=1-5]
      \arrow["\sigma", bend left=-30, from=1-4, to=1-3]
   \end{tikzcd}
\end{align}
Here, $A(P)=\frac{TP}G$ is the Atiyah algebroid of $P\ra M$, $\ad (P)=\frac{P\times \frak g}{G}$ is the adjoint vector bundle, and $\rho$ is induced by $\d\pi$. The interpretation here is that a splitting $\sigma$ is the same as a horizontal lift of vector fields on $M$, to $G$-invariant ones on $P$. Since \eqref{eq:atiyah_ses} is a short exact sequence of Lie algebroids, any splitting $\sigma$ has a \textit{curvature} $F^\sigma$, measuring the failure of $\sigma$ to be a Lie algebroid morphism:
\[
F^\sigma\in\Omega^2(M;\ad P),\quad F^\sigma(X,Y)=\sigma([X,Y])-[\sigma(X),\sigma(Y)].
\]
Moreover, a splitting $\sigma$ induces a connection $\nabla^\sigma$ on $\ad P$, given by  $\nabla^\sigma_X \xi = [\sigma(X),\xi],$ for all  $X\in \vf(M)$ and $\xi\in\Gamma(\ad P)$. Importantly, Jacobi identity on $A(P)$ implies that the triple $(\sigma, \nabla^\sigma, F^\sigma)$ satisfies:
\begin{enumerate}[label={(\roman*)}]
   \item The connection $\nabla^\sigma$ preserves the Lie bracket on $\ad (P)$: $\nabla^\sigma[\xi,\eta]=[\nabla^\sigma\xi,\eta]+[\xi,\nabla^\sigma\eta]$.
   \item The curvature tensor $R^{\nabla^\sigma}$ of $\nabla^\sigma$ equals the adjoint action by $F^\sigma$, i.e., $R^{\nabla^\sigma} = [-, F^\sigma]$,
   \item The curvature satisfies the Bianchi identity, $\d{}^{\nabla^\sigma} F^\sigma=0$.
\end{enumerate}

Now, to formulate a variational problem, we need to additionally pick a metric (and orientation) on the base $M$, as well as an $\ad$-invariant bundle metric on $\ad(P)$. 
This allows us to define the action functional \eqref{eq:classical_YM_action}. 
Noting that the space of splittings is an affine space modelled on $\Omega^1(M;\ad P)$, varying this action yields precisely the Yang--Mills equation \eqref{eq:classical_YM_eqn}. 

This approach readily generalizes to (possibly non-integrable) transitive Lie algebroids $A\Ra TM$, by simply considering the splittings of the SES defined by the anchor $\rho\colon A\twoheadrightarrow TM$, sometimes called the abstract Atiyah sequence. Conceptually, this offers an important insight: Yang--Mills theory is fundamentally \textit{infinitesimal in nature}, since its construction does not depend on global data. The more difficult part is in extending this to \textit{non-transitive} Lie algebroids, which are moreover possibly  \textit{non-regular}, i.e., the anchor $\rho\colon A\ra TM$ need not be surjective or have constant rank. The main idea of how to proceed in the non-transitive case is explained below.


\SkipTocEntry\subsection*{What are our fields?} 
In the classical case, the dynamical fields (i.e., the variables that the action functional takes as inputs) were simply the splittings of \eqref{eq:atiyah_ses}. To describe what the dynamical fields are in the non-transitive case, we must first set the stage by fixing a short exact sequence of Lie algebroids.
\begin{definition}
   \label{def:boi_algd}
   A \textit{bundle of ideals} of a Lie algebroid $A\Ra M$ is a vector subbundle $\frak k$ of $A$ contained within its isotropy ($\frak k\subset \ker\rho$), whose sections $\Gamma(\frak k)$ form an ideal for the Lie bracket on $\Gamma(A)$: 
\begin{align}
   [\Gamma(A),\Gamma(\frak k)]\subset \Gamma(\frak k).
\end{align}
\noindent In this way, $\frak k\ra M$ is naturally a bundle of Lie algebras with the bracket inherited from $A$, and we have the following associated short exact sequence of Lie algebroids over $M$:
 \begin{align}
     \label{eq:ses_bdl_ideals}
   \begin{tikzcd}[ampersand replacement=\&]
  	0 \& \frak k \& A \& B \& 0,
  	\arrow[from=1-1, to=1-2]
  	\arrow[from=1-2, to=1-3]
  	\arrow["\phi", from=1-3, to=1-4]
  	\arrow[from=1-4, to=1-5]
  \end{tikzcd}
 \end{align}
 where $\phi\colon A\ra B=A/\frak k$ is the quotient projection. Alternatively, one can start with a surjective Lie algebroid map $\phi\colon A\ra B$, covering $\id_M$, and define $\frak k=\ker\phi$. In both cases, the Lie algebroid $A$ is called an \textit{extension} of the Lie algebroid $B$ by a bundle of Lie algebras $\frak k$. Note that both Lie algebroids $A$ and $B$ determine the same (singular) orbit foliation on $M$, which we denote by $\F$. 
\end{definition}

In the setting of general Lie algebroid extensions, the aforementioned \textit{IM connections} provide a natural, geometrically richer notion of a connection compared to just vector bundle splittings of \eqref{eq:ses_bdl_ideals}. They have a rather involved definition which we postpone until \sec\ref{sec:background_imecs}; in supergeometric terms, they are connections on the $Q$-bundle $\pi^*\frak k[1]\ra A[1]$, see Remark \ref{rem:supergeometry}. For the purpose of this introduction, we will think of them in a down-to-earth way: as triples $(\sigma, \nabla, F)$, satisfying the compatibility relations below, extending the conditions (i--iii) we saw in the transitive case. These triples will be our dynamical fields. Although this approach is very concrete, it comes with the price of obscuring why these are the \enquote{natural} fields to consider, which will become clear in \sec\ref{sec:preliminaries}.
  
 \begin{definition}
\label{def:gauge_triples}
   Given a bundle of ideals $\frak k$ of a Lie algebroid $A\Ra M$, a \textit{connection triple} $(\sigma,\nabla,F)$ consists of a vector bundle splitting $\sigma\colon B\ra A$ of the sequence \eqref{eq:ses_bdl_ideals}, a connection $\nabla$ on $\frak k\ra M$, and a 2-form $F\in\Omega^2(M;\frak k)$. They must satisfy the following compatibility conditions:
{\setlength\multicolsep{0.5em}
\begin{multicols}{2}
\begin{enumerate}[itemsep=0.5em, topsep=0pt]
  \item $\nabla$ preserves the Lie bracket on $\frak k$.
  \item $R^\nabla=[-,F]_{\frak k}$.
  \item The \textit{curvature 3-form} $G\coloneq\d{}^\nabla F$ \\is transversal: $\iota_{X}G=0$ for all $X\in T\F$.
  \item $\nabla_{\rho_B(\beta)}=[\sigma(\beta),-]$, for any $\beta\in B$.
  \item $F^\sigma=(\rho_B)^*F$.
  \item[\vspace{\fill}] \phantom{}\\\phantom{}
\end{enumerate}
\end{multicols}}
\vspace{-0.2em}
\noindent Equivalently, it consists of an IM connection together with a curving $F$; see Definitions \ref{def:im_connection} and \ref{def:curving}.
 \end{definition}
Comparing to the transitive case, the \textit{Bianchi identity} in condition (iii) now clearly stands out: it no longer reads $\d{}^\nabla F=0$. Instead, it implies that \textit{the curvature 3-form $G$ is an $A$-invariant form} (Definition \ref{def:invariant_forms}), meaning we regard it as a form on the (possibly non-smooth) leaf space $M/\F$. This leads to a conceptual shift regarding the curvature, now encoded by a \textit{curvature pair} $(F,G)$.  The fact that $F$ is a part of the dynamical field, as opposed to being derived from the IM connection, reflects another paradigm shift: $F$ plays the role of a \enquote{potential}, modelling an abstract object --- the abstract curvature of an IM connection, which is an IM 2-form on $A$, see Definition \ref{def:curvature}.

At last, the conditions (iv) and (v) appearing above state that leafwise (i.e., tangentially to $\F$), $\nabla$ and $F$ are completely determined by the splitting $\sigma$.
Here, $\rho_B\colon B\ra TM$ denotes the anchor of the algebroid $B\Ra M$, and $F^\sigma$ denotes the curvature of the splitting $\sigma$, defined analogously as in the transitive case:
\begin{align}
\begin{split}
      \label{eq:curvature_splitting}
      &F^\sigma\in \Omega^2(B;\frak k)\coloneq\Gamma(\Lambda^2B^*\otimes \frak k),\\ &F^\sigma(\beta_1,\beta_2)=\sigma[\beta_1,\beta_2]-[\sigma(\beta_1),\sigma(\beta_2)].
\end{split}
\end{align}

\SkipTocEntry\subsection*{Generalized Yang\texorpdfstring{--}{-}Mills theory as a constrained variational problem}
When constructing an action functional for connection triples $(\sigma, \nabla, F)$, the idea is that besides accounting for the norm of the 2-curvature $F$, the norm of the corresponding 3-curvature $G=\d{}^\nabla F$ should now also contribute to the action. With this in mind, we can state our main theorem. The groupoid version of this theorem is given later, in Theorem \ref{thm:main_theorem_groupoids}.

\begin{theorem}
      \label{thm:main_theorem}
      Let $\frak k$ be a bundle of ideals of a Lie algebroid $A$ over a compact, oriented, pseudo-Riemannian manifold $M$, and let $\inner\cdot\cdot_{\frak k}$ be an ad-invariant bundle metric on $\frak k$. For any connection triple $(\sigma,\nabla,F)$, the following statements are equivalent.
      \begin{enumerate}[label={(\roman*)}]
         \item The triple $(\sigma,\nabla,F)$ is a solution to the \textbf{constrained variational problem} for the functional 
   \begin{align}
      \label{eq:ym_action}
      \S(\sigma,\nabla,F)=\int_M \inner FF_{\frak k}\vol_M+\mu\int_M \inner GG_{\frak k}\vol_M,\quad (\mu\in\R\text{ fixed}),
   \end{align}
      subject to the constraint that the cohomological class, defined by the underlying IM connection of the triple, is constant in the cohomology of the Weil complex of the Lie algebroid $A$.
      \item The curvature pair $(F,G)$ satisfies the following Euler--Lagrange equation and relation,
      \begin{align}
         \label{eq:YM.I}
            (\d{}^\nabla)^* F&=0,\tag{YM.I}\\
   \label{eq:YM.II}
            F+\mu\,(\d{}^\nabla)^* G &\perp \Omega^2_{\ainv}(M;\frak k),\tag{YM.II}
      \end{align}
      which are systems of nonlinear partial differential equations and relations, respectively.
      \end{enumerate}
\end{theorem}
  \noindent The subscript \enquote{$A$-inv} in \eqref{eq:YM.II} denotes \textit{$A$-invariant forms}; as already mentioned, they serve as a model for the forms on the (possibly non-smooth) leaf space $M/\F$. 
  The differential operator $(\d{}^\nabla)^*$ denotes the covariant coderivative, i.e., the formal adjoint to $\d{}^\nabla$ with respect to the $L^2$-pseudo inner product on $\Omega^\bullet(M;\frak k)$, induced by integrating over $M$. Compactness of $M$ is assumed so that the integral \eqref{eq:ym_action} is finite; alternatively, we can work with compactly supported forms.
  

\SkipTocEntry\subsection*{Gauge invariance of solutions} 
Given a variational problem, \textit{gauge invariance} (or \textit{automorphism invariance}) refers to the principle that any automorphism of the underlying structure which defines the variational problem, should also preserve its solution space. 
Such an automorphism that preserves the underlying structure will be called a \textit{gauge transformation}. More precisely, in our setting we have:
\begin{definition}
   \label{def:infinitesimal_gauge_transformations}
   Assume the setting of Theorem \ref{thm:main_theorem}.\ An \textit{(infinitesimal) gauge transformation} is a Lie algebroid automorphism $\Theta$ of $A\Ra M$ covering an orientation-preserving isometry on the base $M$, restricting on $\frak k\subset A$ to an isometry for the bundle metric $\inner\cdot\cdot_{\frak k}$. 
\end{definition}
Note that any such $\Theta$ is in particular an automorphism of the extension \eqref{eq:ses_bdl_ideals}, since we are requiring $\Theta(\frak k)= \frak k$. Gauge transformations thus naturally act on connection triples from Definition \ref{def:gauge_triples}, and the next theorem asserts this action preserves the solution space, i.e., gauge invariance.
\begin{theorem}
   \label{thm:gauge_invariance}
   In the setting of Theorem \ref{thm:main_theorem}, the action functional \eqref{eq:ym_action} is invariant under infinitesimal gauge transformations, hence the solution space to \eqref{eq:YM.I} and \eqref{eq:YM.II} is preserved by them. 

   In particular, this holds for flows of inner derivations $[\alpha,\cdot]$ of the Lie algebroid $A$, induced by the sections $\alpha\in\Gamma(A)$ such that the flow of $\rho(\alpha)$ is an orientation-preserving isometry of the base $M$.
\end{theorem}
   For intuition, let us relate back to principal bundles. For an integrable algebroid $A$, a class of infinitesimal gauge transformations is given by adjoint actions of bisections of an integrating (source-fibre connected) Lie groupoid $\G$, that is, differentials $\Ad_\sigma=(I_\sigma)_*\colon A\ra A$ of inner automorphisms $I_\sigma\colon \G\ra \G$, where $\sigma$ is a bisection of $\G$. According to the definition above, the base map of $\sigma$ must be an orientation-preserving isometry; for instance, this holds for static bisections, i.e., those bisections $\sigma$ that take values in the isotropy of $\G$. In the classical case, the usual principal bundle automorphisms are precisely the static bisections of its gauge groupoid.

\SkipTocEntry\subsection*{Instantons in 4D and 5D}
The classical (anti) self-duality relation of a 2-form on a 4-dimensional base $M$ reads
\begin{align}
   \label{eq:asd_4}
   F=\pm \star F.
\end{align}
For a transitive algebroid $A$, if the curvature $F$ of a splitting $\sigma$ of the abstract Atiyah sequence satisfies \eqref{eq:asd_4}, Bianchi's identity $\d{}^{\nabla}F=0$ implies $F$ is indeed a solution to \eqref{eq:YM.I}, and the PDR \eqref{eq:YM.II} is vacuously fulfilled. We recover the usual (anti) self-dual principal connections when $A$ is moreover integrable. Interestingly, however, the transitive case is not the only situation where this applies: for instance, when the bundle of ideals $\frak k\subset$ has a semisimple typical fibre, regardless of transitivity of $A$, we similarly get that any connection triple $(\sigma,\nabla,F)$ whose 2-curvature $F$ satisfies \eqref{eq:asd_4}, is automatically a solution of the variational problem \eqref{eq:ym_action}, and the  constraint is vacuous.

Our generalization to the non-transitive setting allows us to analogously introduce a special class of solutions when the base $M$ is 5-dimensional. In this case, the \textit{(anti) self-dual} connection triples are defined by the condition that the curvature pair $(F,G)$ satisfies
  \begin{align}
   \label{eq:asd_5}
    G=\pm \star F.
  \end{align}
Bianchi's identity implies such triples are automatically solutions to \eqref{eq:YM.I} and \eqref{eq:YM.II}, if we let $\mu=(-1)^s$ where $\mu$ is the constant from Theorem \ref{thm:main_theorem} and $s$ is the index of the pseudo Riemannian metric on $M$. We will see that these special solutions may only exist in very special circumstances.

\subsubsection*{Acknowledgments} I am deeply grateful to Rui L.\ Fernandes for his interest in this project, and the many insightful conversations. This project was a part of my PhD, and my deepest thanks go to my advisers Ioan Mărcuţ and Pedro Resende for their guidance and advice. I also thank Alejandro Cabrera, Marius Crainic, João Nuno Mestre, and Lennart Obster for the many helpful discussions.

\numberwithin{theorem}{section}

\section{Preliminaries}
\label{sec:preliminaries}
The two main sources for the prerequisite background on (infinitesimal) MECs are \cite{mec} and \cite{covariant_derivatives}. In this section, we recall the global and infinitesimal notions of MECs, together with their curvature, and note their fundamental properties. We will also recall the Bott--Shulman--Stasheff complex of forms, and Weil complex in the infinitesimal setting, since these complexes are central to our theory. 
We begin with Lie groupoids, since they are more geometric and intuitive, whereas the infinitesimal setting discussed in \sec\ref{sec:background_imecs} is more (Lie) algebraic and computational. 
\subsection{Multiplicative Ehresmann connections on Lie groupoids}
We begin by describing bundles of ideals on a Lie groupoid, the global analogue of Definition \ref{def:boi_algd}. 
\begin{definition}
\label{def:boi}
On a Lie groupoid $\G\rra M$, a vector bundle $\frak k\subset \ker \rho$ is called a \textit{bundle of ideals}, if for any $g\in \G$, the map $\Ad_g\colon \ker\rho_{s(g)}\ra\ker\rho_{t(g)}$ given by $\Ad_g=\d({C_g})$, restricts to a map
\[\Ad_g\colon \frak k_{s(g)}\ra \frak k_{t(g)}.\] 
Here, $\rho$ denotes the anchor of the Lie algebroid of $\G$. This definition makes $(\frak k,\Ad)$ into a representation of $\G$, which we will sometimes denote by $\Ad\colon \G\curvearrowright \frak k$.
\end{definition}
   Any bundle of ideals on $\G$ is also a bundle of ideals on its Lie algebroid as in Definition \ref{def:boi_algd}, and the converse holds if $\G$ is $s$-connected \cite{rigidity_poisson_submanifolds}. The main class of examples of bundles of ideals on Lie groupoids comes from surjective submersive groupoid morphisms $\Phi\colon\G\ra \H$ covering the identity, 
\[\begin{tikzcd}[column sep=small]
	\G && \H \\
	& M
	\arrow["\Phi", from=1-1, to=1-3]
	\arrow[shift right, from=1-1, to=2-2]
	\arrow[shift left, from=1-1, to=2-2]
	\arrow[shift left, from=1-3, to=2-2]
	\arrow[shift right, from=1-3, to=2-2]
\end{tikzcd}\]
   where one takes as $\frak k$ the kernel of the associated Lie algebroid morphism, $\frak k=\ker\d\Phi|_M$. In this case, we are in the situation where $\G$ is an extension of $\H$ by a bundle of normal Lie subgroups of the isotropy groups of $\G$. That is, we have the following SES of Lie groupoids:
\[\begin{tikzcd}
	1 & {\ker\Phi} & \G & \H & 1,\quad \text{where }\ker\Phi=\Phi^{-1}(1_M).
	\arrow[from=1-1, to=1-2]
	\arrow[from=1-2, to=1-3]
	\arrow["\Phi", from=1-3, to=1-4]
	\arrow[from=1-4, to=1-5]
\end{tikzcd}\]

Contrary to the case of algebroids, not all bundles of ideals arise in this way; given a bundle of ideals $\frak k$ on $\G\rra M$, we define its \textit{smearing} as the distribution $K\subset T\G$,
  \begin{align}
    \label{eq:smearing}
    K_g\coloneqq \d (L_g)_{1_{s(g)}}(\frak k_{s(g)})=\d(R_g)_{1_{t(g)}}(\frak k_{t(g)}).
  \end{align}
  It is easy to see $K$ is involutive; if the corresponding foliation $\F(K)$ on $\G$ is simple, the natural projection $\Phi\colon \G\ra \G/\F(K)$ is a surjective submersive Lie groupoid morphism with $\ker\d\Phi|_M=\frak k$.
\begin{definition}
   \label{def:mec}
   A \textit{multiplicative Ehresmann connection} (MEC) for a bundle of ideals $\frak k$ on a Lie groupoid $\G\rra M$, is equivalently described as either:
\begin{itemize}
   \item A distribution $E\subset T\G$ which is a Lie subgroupoid of $T\G\rra TM$, satisfying $T\G=E\oplus K$.
   \item A multiplicative 1-form $\omega\in \Omega^1(\G;\frak k)$ with values in the representation $\frak k$, satisfying $\omega|_{\frak k}=\id_{\frak k}$.
\end{itemize}
(See below for the definition of multiplicativity of a differential form.) The set of all MECs is denoted
\begin{align*}
  \A(\G;\frak k)=\set{\omega\in \Omega^1_m(\G;\frak k)\given \omega|_{\frak k}=\id_{\frak k}}.
\end{align*}
One should view the connection 1-form $\omega$ as the vertical projection $T\G\ra K\cong s^*\frak k$. The relation between the two descriptions above is $E=\ker\omega$, and we keep in mind that roughly speaking, there is a correspondence of the conditions: \textit{Lie subgroupoid $\leftrightarrow$ multiplicative}, and $T\G=E\oplus K \leftrightarrow \omega|_{\frak k}=\id_{\frak k}$.
\end{definition}
\begin{remark}
   If $\frak k$ comes from a surjective submersive Lie groupoid morphism $\Phi\colon \G\ra \H$ (as above), there holds $K=\ker\d\Phi$, so $E$ is an Ehresmann connection with the additional property of being a Lie subgroupoid of $T\G$. Being a subgroupoid just means that $E$ is closed under the partial multiplication $\d m$ on $T\G\rra TM$, and we stress that this does not imply involutivity of $E\subset T\G$.
\end{remark}
\begin{definition}
   \label{def:multiplicative}
   A differential form $\omega\in\Omega^k(\G;\frak k)\coloneq \Omega^k(\G;s^*\frak k)$ is \textit{multiplicative}, if there holds
\begin{align}
\label{eq:multiplicative}
  \omega_{gh}(\d m(X_i,Y_i))_i=\omega_h(Y_i)_i+\Ad_{h^{-1}}\circ \omega_g(X_i)_i
\end{align}
for any composable arrows $g,h\in \G$ and any $k$ composable pairs of vectors $(X_i,Y_i)\in T_{(g,h)}\G^{(2)}$, i.e., $\d s_g(X_i)=\d t_h(Y_i)$. The space of all multiplicative forms is denoted by $\Omega^k_m(\G;\frak k)$.
\end{definition}
\subsubsection{Complexes of forms}\label{sec:intro_bss_complex} There is an important cochain complex in which multiplicative $k$-forms are precisely the 1-cocycles \cite{diff_cohomology}. It is the \textit{Bott--Shulman--Stasheff} complex of differential forms on the nerve of $\G\rra M$ (for a fixed form degree $k$), with values in the representation $\frak k$,
\begin{align}
   \label{eq:bss}
\begin{tikzcd}[ampersand replacement=\&, column sep=large, row sep=large]
	{\Omega^{k}(M;\frak k)} \& {\Omega^{k}(\G;\frak k)} \& {\Omega^{k}(\G^{(2)}; \frak k)} \& \cdots
	\arrow["{\delta^0}", from=1-1, to=1-2]
	\arrow["{\delta^1}", from=1-2, to=1-3]
	\arrow["{\delta^2}", from=1-3, to=1-4]
\end{tikzcd}
\end{align}
   where $\delta$ is the simplicial differential.  Apart from multiplicative forms, another class of forms appearing in this complex also plays a big role in our framework:  the 0-cocycles, known as \textit{invariant forms} on the base manifold. They are the forms which are in the kernel of differential at zeroth level,
\begin{align}
\begin{split}
      \label{eq:invariant}
      &\delta^0\colon \Omega^k(M;\frak k)\ra \Omega^k(\G;\frak k),\\
      &(\delta^0\gamma)_g= (s^*\gamma)_g - \Ad_{g^{-1}}\circ(t^*\gamma)_g,\quad (g\in \G).
\end{split}
\end{align}
The multiplicative forms arising as coboundaries, $\im(\delta^0)$, are called \textit{multiplicatively exact}, or sometimes also \textit{basic} or \textit{cohomologically trivial}, since their classes vanish in the cohomology of \eqref{eq:bss}.

Now, if we fix a MEC $\omega$, it induces a differential operator, called the \textit{horizontal exterior covariant derivative}, generalizing the well-known operator from principal bundles. It is defined as the map
\[
\D{}^\omega\colon \Omega^k(\G^{(\bullet)};\frak k)\ra \Omega^{k+1}(\G^{(\bullet)};\frak k), \quad \D{}^\omega=h^*\circ\d{}^\nabla.
\]
Here, $\nabla$ denotes the linear connection on the vector bundle $\frak k\ra M$ induced by the given MEC $\omega$ (see \cite{covariant_derivatives}*{Proposition 3.9} for a formula), and $h^*$ denotes the precomposition with the horizontal projection $h\colon T\G\ra E\hookrightarrow T\G$.
Importantly, it turns out $\D{}^\omega$ is a cochain map (Theorem 3.18 in \textit{loc.\,cit.}), meaning we get the following commutative diagram at any $k\geq 0$.
\begin{align}
\label{eq:bss_ideals}
\begin{tikzcd}[ampersand replacement=\&, column sep=large, row sep=large]
	{\Omega^{k+1}(M;\frak k)} \& {\Omega^{k+1}(\G;\frak k)} \& {\Omega^{k+1}(\G^{(2)}; \frak k)} \& \cdots \\
	{\Omega^k(M;\frak k)} \& {\Omega^{k}(\G;\frak k)} \& {\Omega^k(\G^{(2)};\frak k )} \& \cdots
	\arrow["{\delta}", from=1-1, to=1-2]
	\arrow["{\delta}", from=1-2, to=1-3]
	\arrow["{\delta}", from=1-3, to=1-4]
	\arrow["{\d{}^\nabla}", from=2-1, to=1-1]
	\arrow["{\delta}", from=2-1, to=2-2]
	\arrow["{\D{}^\omega}", from=2-2, to=1-2]
	\arrow["{\delta}", from=2-2, to=2-3]
	\arrow["{\D{}^\omega}", from=2-3, to=1-3]
	\arrow["{\delta}", from=2-3, to=2-4]
\end{tikzcd}
\end{align}
\subsubsection{Curvature of a MEC} The operator $\D{}^\omega$ in diagram \eqref{eq:bss_ideals} squares to zero if and only if the distribution $E=\ker\omega$ is involutive. As usual, involutivity is controlled by a curvature 2-form.
\begin{definition}
   \label{def:mec_curvature}
   The \textit{curvature} of a MEC $\omega\in \A(\G;\frak k)$ is the form $\Omega^\omega\in\Omega^2_m(\G;\frak k)$, defined as
   \[
   \Omega^\omega=\D{}^\omega\omega,\quad \text{or more explicitly,}\quad \Omega^\omega(X,Y)=(\d{}^{s^*\nabla}\omega)(h(X),h(Y)).
   \]
\end{definition}
By the middle square in diagram \eqref{eq:bss_ideals}, since the form $\omega$ is multiplicative (and hence, a cocycle), the curvature $\Omega^\omega$ is automatically a multiplicative form as well. Additionally, it satisfies the following properties, see \cite{mec}*{Proposition 2.24} or \cite{covariant_derivatives}*{Proposition 5.7}.
\begin{samepage}
\begin{enumerate}
   \item $\Omega^\omega =0$ if and only if $E=\ker\omega$ is involutive.
   \item The structure equation holds:
$
   \Omega^\omega=\d{}^{\nabla}\omega+\frac 12 [\omega,\omega]_{\frak k}.
$
   \item It satisfies the \textit{Bianchi identity}:
$
      \D{}^\omega \Omega^\omega = 0.
$
\end{enumerate}
\end{samepage}

\subsection{Infinitesimal multiplicative Ehresmann connections on Lie algebroids}
\label{sec:background_imecs}
To define IM connections on Lie algebroids, we use the approach with differential forms. Therefore, we first have to know what is the infinitesimal analogue of multiplicative forms from Definition \ref{def:multiplicative}. These abstract objects were found in \cite{spencer} for values in arbitrary representations, which is what we need. The research on infinitesimal multiplicative forms, with real or general coefficients, remains very active \cites{twisted_dirac, linear_mult, weil, im_forms, ve_mein, multiplicative_tensors, vb-valued, local, fat_lie_theory}.

\begin{definition}
\label{def:im_forms}
Given a bundle of ideals  $\frak k$ of a Lie algebroid $A$, a $\frak k$-valued \textit{IM form} (of degree $k$) on $A$ is a pair of maps $(L,l)$, where 
\begin{align*}
  L\colon \Gamma(A)\ra \Omega^k(M;\frak k) \text{ is linear},
\quad\quad
  l\colon \Gamma(A)\ra \Omega^{k-1}(M;\frak k)\text{ is $C^\infty(M)$-linear},
\end{align*}
and the failure of $C^\infty(M)$-linearity of $L$ is captured by the \textit{Leibniz identity},
\[
L(f\alpha)=f L(\alpha)+\d f\wedge l(\alpha).
\]
The pair $(L,l)$ must moreover satisfy the following \textit{IM conditions}, which are viewed as compatibility conditions with the Lie bracket on $A$:
\begin{align}
  L[\alpha,\beta]&=\L^A_{\alpha} L(\beta)-\L^A_{\beta} L(\alpha),\label{eq:c1}\tag{C.1}\\
  l[\alpha,\beta]&=\L^A_{\alpha} l(\beta)-\iota_{\rho(\beta)}L(\alpha),\label{eq:c2}\tag{C.2}\\
  \iota_{\rho(\alpha)} l(\beta)&=-\iota_{\rho(\beta)} l(\alpha).\label{eq:c3}\tag{C.3}
  \end{align}
  \noindent Here, the Lie derivative $\L_\alpha^A$ acting on $\Omega^k(M;\frak k)$ is defined for each $\alpha\in\Gamma(A), \gamma\in\Omega^k(M;\frak k)$ by
\begin{align*}
   (\L_\alpha^A \gamma)(X_i)_i=[\alpha, \gamma(X_i)_i] - \textstyle\sum_i \gamma(X_1,\dots,[\rho(\alpha),X_i],\dots,X_k).
\end{align*}
 Given an IM form $(L,l)$, the map $L$ is called the \textit{principal part}, and the bundle map $l$ its \textit{symbol}. The space of all IM forms is denoted $\Omega^k_{im}(A;\frak k)$, following the notation for forms on groupoids.
\end{definition}
\begin{definition}[\cite{mec}]
   \label{def:im_connection}
An \textit{infinitesimal multiplicative} (IM) \textit{connection} for a bundle of ideals $\frak k\subset A$ is a $\frak k$-valued IM 1-form, usually denoted by $(\C,v)$, whose symbol $v\colon A\ra \frak k$ satisfies $v|_{\frak k}=\id_\frak k.$ In other words, its symbol is a (left) splitting of the SES \eqref{eq:ses_bdl_ideals}, 
\begin{align}
\label{eq:splitting}
  \begin{tikzcd}[ampersand replacement=\&]
	0 \& {\frak k} \& A \& {B} \& 0.
	\arrow[from=1-1, to=1-2]
	\arrow[from=1-2, to=1-3]
	\arrow["v", bend left=40, from=1-3, to=1-2]
	\arrow[from=1-3, to=1-4]
	\arrow[from=1-4, to=1-5]
\end{tikzcd}
\end{align}
The space of all IM connections for a bundle of ideals $\frak k$ is denoted $\A(A;\frak k)$.
\end{definition}

To unpack this definition, following \cite{mec}, note that the splitting $A=\frak k\oplus H$, where $H=\ker v$, enables us to split the information contained within $\C\colon\Gamma(A)\ra\Omega^1(M;\frak k)$, into the following objects:
\begin{itemize}
   \item A linear connection $\nabla$ on the bundle of ideals $\frak k\ra M$, given by $\nabla=\C|_{\Gamma(\frak k)}$,
   \item A tensor $U\in \Gamma(H^*\otimes T^*M\otimes \frak k)$, given by $U=\C|_{\Gamma(H)}$ (see Definition \ref{def:curvature} for its meaning).
\end{itemize}
The IM conditions on $\C$ can then be rephrased entirely in terms of $\nabla$ and $U$, but we shall omit them (see \textit{loc.\,cit.}, Proposition 5.11). Thus, an IM connection can itself be viewed as a triple $(v,\nabla, U)$, which is already similar to connection triples from Definition \ref{def:gauge_triples}. A precise relation will be given in Definition \ref{def:curving}, but we need to first inspect the infinitesimal analogue of the complex from \sec\ref{sec:intro_bss_complex}.
\subsubsection{Complexes of forms}
As in the global picture, there is now a cochain complex in which IM forms are precisely 1-cocycles. It is called the \textit{Weil complex} (of cochains with fixed \textit{degree} $k$), 
\begin{align}
   \label{eq:weil_complex}
   \begin{tikzcd}[ampersand replacement=\&, column sep=large, row sep=large]
   {\Omega^{k}(M;\frak k)} \& {W^{1,k}(A;\frak k)} \& {W^{2,k}(A;\frak k)} \& \cdots
   \arrow["{\delta^0}", from=1-1, to=1-2]
   \arrow["{\delta^1}", from=1-2, to=1-3]  
   \arrow["{\delta^2}", from=1-3, to=1-4]
\end{tikzcd}
\end{align}
where $\delta$ denotes the Koszul differential of Weil cochains. The definition of Weil cochains and the differential is nontrivial, so we refer the reader to \cite{homogeneous}*{\sec 4} and \cite{covariant_derivatives}*{\sec 2.2}. This complex is fundamental to our theory, but we only need its description at zeroth and first level. Namely, the 0-cochains are already defined as $\Omega^k(M;\frak k)$, and 1-cochains are 
\[
W^{1,k}(A;\frak k)=\set{\text{pairs $(L,l)$ as in Definition \ref{def:im_forms}, without IM conditions}}.
\]
The Koszul differential at zeroth level is the infinitesimal analogue of equation \eqref{eq:invariant}. It reads
\begin{align}
   \begin{split}
      \label{eq:delta0}
      \delta^0\colon\Omega^k(M;\frak k)\ra \Omega^k_{im}(A;\frak k)
      ,\quad
      \big(\delta^0\gamma\big)(\alpha)=\big(\L_\alpha^A\gamma,\,\iota_{\rho(\alpha)}\gamma\big).
   \end{split}
\end{align}
\vspace{-1.5em}
\begin{definition}
   \label{def:invariant_forms}
   A $k$-form $\gamma\in \Omega^k(M;\frak k)$ is \textit{$A$-invariant}, if it is a 0-cocycle of the Weil complex $W^{\bullet,k}(A;\frak k)$. We denote the space of invariant $k$-forms by
   $
   \Omega_{\ainv}^k(M;\frak k)=\ker(\delta^0)\cap \Omega^k(M;\frak k).
   $
   Moreover, an IM form $(L,l)\in\Omega^k_{im}(A;\frak k)$ is called \textit{cohomologically trivial} (sometimes also \textit{multiplicatively exact} or \textit{basic}), if it is a 1-coboundary, i.e., if $(L,l)=\delta^0\gamma$ for some $\gamma\in\Omega^k(M;\frak k)$.
\end{definition}
\begin{remark}
   \label{rem:invariant_transversal_central}
   Any invariant form $\gamma$ must necessarily have values in the centre $z(\frak k)\subset \frak k$, and must necessarily be \textit{transversal}, i.e., $\iota_X\gamma=0$ for all $X\in T\F$. This follows directly from \eqref{eq:delta0}.
\end{remark}

For the definition of the Koszul differential at higher levels $\delta^{p\geq 1}$, see e.g., \cite{covariant_derivatives}*{eq.\ 2.9}. For instance, at the first level, $\delta^1$ is defined in such a way that $\ker(\delta^1)=\set{\frak k\text{-valued IM forms on }A}$.

\begin{remark}
   At higher levels, the Weil complex also carries important information. For example, the space $H^2(W^{\bullet,1}(A;\frak k)^\Hor)$ contains the obstruction class to the existence of IM connections, where \enquote{$\Hor$} denotes the so-called horizontal subcomplex of the Weil complex; see \cite{covariant_derivatives}*{\sec 5.1} for details.
\end{remark}

Now, similarly as with MECs on Lie groupoids, if we fix an IM connection $(\C,v)$, it induces a differential operator on the Weil complex, again called the \textit{horizontal exterior covariant derivative}. This operator was discovered in \cite{covariant_derivatives}*{\sec 4.3}. It is the map
\[
\D{}^{(\C,v)}\colon W^{\bullet,k}(A;\frak k)\ra W^{\bullet,k+1}(A;\frak k),
\] 
defined by $\D{}^{(\C,v)}=h^*\circ \d{}^\nabla$ as for groupoids, but now $\d{}^\nabla$ and $h^*$ have highly nontrivial expressions which we omit; see, e.g., formula (4.15) in \textit{loc.\,cit.}\  for a precise definition on 1-cochains. We will not need the formula for $\D{}^{(\C,v)}$, but we will need its properties. The most important property is that the following diagram commutes for all $k\geq 0$ (see Theorem 4.13 in \textit{loc.\,cit.}).
    \begin{align}
    \label{eq:weil_ideals}
    \begin{tikzcd}[ampersand replacement=\&, column sep=large, row sep=large]
        {\Omega^{k+1}(M;\frak k)} \& {W^{1,k+1}(A;\frak k)} \& {W^{2,k+1}(A;\frak k)} \& \cdots \\
        {\Omega^k(M;\frak k)} \& {W^{1,k}(A;\frak k)} \& {W^{2,k}(A;\frak k)} \& \cdots
        \arrow["{\delta}", from=1-1, to=1-2]
        \arrow["{\delta}", from=1-2, to=1-3]
        \arrow["{\delta}", from=1-3, to=1-4]
        \arrow["{\d{}^\nabla}", from=2-1, to=1-1]
        \arrow["{\delta}", from=2-1, to=2-2]
        \arrow["{\D{}^{(\C,v)}}", from=2-2, to=1-2]
        \arrow["{\delta}", from=2-2, to=2-3]
        \arrow["{\D{}^{(\C,v)}}", from=2-3, to=1-3]
        \arrow["{\delta}", from=2-3, to=2-4]
    \end{tikzcd}
    \end{align}
\begin{remark}
   \label{rem:supergeometry}
   In supergeometric terms, $\frak k$-valued Weil cochains are the differential forms on the $Q$-manifold $A[1]$, valued in the $Q$-bundle $\pi^*\frak k[1]\ra A[1]$  \cites{mehta_supergeometry, abc} ($\pi\colon A\ra M$ is the projection). The sections of this $Q$-bundle are $\Omega^\bullet(A;\frak k)=\Gamma(\wedge^\bullet A^*\otimes \frak k)$, which are the spaces $W^{\bullet,0}(A;\frak k)$ in the Weil complex at $k=0$, whereas the 1-forms are $W^{\bullet,1}(A;\frak k)$. By viewing IM connections as differential operators $\D{}\colon\Omega^\bullet(A;\frak k)\ra W^{\bullet, 1}(A;\frak k)$, they can be thought of as connections on this $Q$-vector bundle. In this language, IM conditions translate to compatibility with homological vector fields. 
\end{remark}
\subsubsection{Curvature of an IM connection} As with Lie groupoids, it turns out that the diagram \eqref{eq:weil_ideals} portrays a double complex, i.e., $\D{}^{(\C,v)}$ squares to zero, if and only if the curvature of $(\C,v)$ vanishes; the latter is defined as the following IM 2-form.
\begin{definition}
   \label{def:curvature}
The \textit{curvature} of an IM connection $(\C,v)$ is the IM 2-form $\Omega^{(\C,v)}\in\Omega^2_{im}(A;\frak k)$, 
\[\Omega^{(\C,v)}=\D{}^{(\C,v)}(\C,v).\]
Explicitly, if $h=\id-v\colon A\ra H\subset A$ is the horizontal projection, its principal part and symbol are
\begin{align}
\begin{split}
   \label{eq:im_curv}
   \Omega^{(\C,v)}(\alpha)&=\big(R^\nabla\cdot v\alpha+\d{}^\nabla U(h\alpha),U(h\alpha)\big).
\end{split}
\end{align}
\end{definition} 
The middle square in diagram \eqref{eq:weil_ideals} implies the curvature is indeed an IM 2-form since $(\C,v)$ is itself a 1-cocycle. It satisfies the \textit{infinitesimal Bianchi identity} \cite{covariant_derivatives}*{Theorem 5.12}:
\begin{align}
   \label{eq:bianchi}
   \D{}^{(\C,v)}\Omega^{(\C,v)}=0.
\end{align}
We see that the tensor $U$ appears as the symbol of the curvature IM 2-form. The IM conditions imply the curvature of the splitting $v$, as in \eqref{eq:curvature_splitting}, is described by $U$ as $F^\sigma(\beta_1,\beta_2)=U(\sigma\beta_1)(\rho\beta_2)$. 

Finally, we can relate the connection triples from Definition \ref{def:gauge_triples} with IM connections. 
\begin{definition}[\cite{covariant_derivatives}*{\sec 5.3}]
   \label{def:curving}
   An IM connection $(\C,v)$ is \textit{primitive} if the cohomological class of its curvature vanishes, $[\Omega^{(\C,v)}]=0$, in the Weil cohomology group $H^1(W^{\bullet, 2}(A;\frak k))$. That is, 
   \[
   \Omega^{(\C,v)}=\delta^0 F.
   \]
   for some 2-form $F\in \Omega^2(M;\frak k)$. Any such 2-form $F$ is called a \textit{curving} of the IM connection $(\C,v).$ In other words, an IM connection is primitive if and only if its curvature is multiplicatively exact.
\end{definition}
A triple $((\C,v),F)$ consisting of an IM connection and a curving $F$, is the same as a connection triple from Definition \ref{def:gauge_triples}: the relation is given by $\C(\alpha)=\nabla(v\alpha)+\iota_{\rho(\alpha)}F$. The left square in diagram \eqref{eq:weil_ideals} shows that the Bianchi identity \eqref{eq:bianchi} is captured by the invariance of the \textit{curvature 3-form} $G\coloneq\d{}^\nabla F$, that is, $\delta^0 G=0$. In particular, $G$ must be transversal and $z(\frak k)$-valued, see Remark \ref{rem:invariant_transversal_central}. 
\begin{remark}
   The terminology for \enquote{curvings} comes from bundle gerbes \cites{bundle_gerbes_original, bundle_gerbes, gerbes}. For us, they are simply a means to model the curvature with a form on the base. For instance, on a principal bundle $P$, the curvature can be viewed as an $\ad(P)$-valued 2-form on the base, thus, on principal bundles, curvings always exist and are unique. As we will see in \sec\ref{sec:ym_gerbes}, curvings also always exist for connections on bundle gerbes (but they are not unique). The adjective \enquote{primitive} is not used derogatorily, but rather describes something which is the first known instance of its kind. 
\end{remark}

We note that a connection triple with $G=0$ exists if and only if the Lie algebroid $A$ is of so-called \textit{principal type}; these are the fibred products of an arbitrary Lie algebroid with a transitive one. See \cite{mec}*{\sec 3.5} for the definition and \cite{covariant_derivatives}*{\sec 5.3.1} for the proof of this characterization.



\section{Generalized Yang--Mills theory}
Since IM connections and curvings are more natural to work with than connection triples from Definition \ref{def:gauge_triples}, we begin by denoting the domain of the action functional \eqref{eq:ym_action} by
  \begin{align*}
    \DD(A;\frak k)=\set*{((\C,v),F)\in\A(A;\frak k)\times \Omega^2(M;\frak k)\given 
      \delta^0F=\Omega^{(\C,v)}
    }.
  \end{align*}
We shall refer to any such triple $((\C,v),F)$ as a connection triple as well. 
To prove Theorem \ref{thm:main_theorem}, we need to vary the action. To this end, it is important to realize that the space $\DD(A;\frak k)$ decomposes into affine spaces determined by the cohomological classes of IM connections in the Weil cohomology. These affine spaces are the constraints of our action functional. In the end, we will vary the functional in the directions tangential to these affine spaces, which amounts to finding solutions to our constrained variational problem.
\subsection{Affine deformations of connection triples}
\label{sec:affine_deformations}
The space of all IM connections $\A(A;\frak k)$ is itself clearly an affine space modelled on the vector space of \textit{horizontal} IM 1-forms, denoted $\Omega^1_{im}(A;\frak k)^\Hor$, consisting of pairs $(L,l)\in \Omega^1_{im}(A;\frak k)$ whose symbol satisfies $l|_{\frak k}=0$.  
Notice that any multiplicatively exact 1-form $(L,l)=\delta^0\gamma$ is clearly horizontal by \eqref{eq:delta0}, and the following shows that deforming by them preserves primitivity of IM connections.
\begin{lemma}
   \label{lemma:affine_primitive}
   The deformation of a primitive IM connection by a multiplicatively exact IM 1-form is again a primitive IM connection. More precisely, if $((\C,v),F)$ is a connection triple, then
\begin{align}
   \label{eq:affine_primitive}
   \big((\C,v)+\delta^0 \gamma, F^\gamma\big),\quad F^\gamma\coloneq F+\d{}^\nabla\gamma-\tfrac 12 [\gamma,\gamma]_{\frak k},
\end{align}
   is also a connection triple, for any $\gamma\in \Omega^1(M;\frak k)$. Moreover, under any such affine deformation, the 3-curvature $G=\d{}^\nabla F$ remains unchanged.
\end{lemma}
\begin{remark}
   \label{rem:gla}
   Regarding the notation for the bracket of $\frak k$-valued forms: the bracket $[\cdot,\cdot]_{\frak k}$ on $\frak k$ induces a graded Lie algebra structure on $\Omega^\bullet(M;\frak k)$, given by $[\cdot,\cdot]_{\frak k}\colon \Omega^k(M;\frak k)\times \Omega^l(M;\frak k)\ra \Omega^{k+l}(M;\frak k)$,
   \[
   [\gamma_1\otimes \xi_1,\gamma_2\otimes \xi_2]_{\frak k}=\gamma_1\wedge \gamma_2\otimes [\xi_1,\xi_2]_{\frak k} \text{ for simple tensors }\gamma_i\otimes \xi_i\in\Omega^\bullet (M;\frak k).
   \]
\end{remark}
\begin{proof}
   This lemma is proved in \cite{covariant_derivatives}*{Proposition 5.30}; we give an idea of proof here. Under any affine deformation $(\C,v)\ra (\C,v)+(L,l)$, where $(L,l)\in\Omega^1_{im}(A;\frak k)^\Hor$, the curvature changes as
      \begin{align}
      \label{eq:expansion_inf}
      \Omega^{(\C,v)+(L,l)}=\Omega^{(\C,v)}+ \D{}^{(\C,v)}(L,l)+\mathbf c_2(L,l).
      \end{align}
      Here, $\mathbf c_2\colon\Omega^1_{im}(A,\frak k)^\Hor\rightarrow \Omega^2_{im}(A,\frak k)^\Hor$ is a certain map which is homogeneous of degree two; see Theorem 5.19 in \cite{covariant_derivatives} for a proof of \eqref{eq:expansion_inf} and the formula of $\mathbf c_2$. Importantly, it satisfies the equality $\mathbf c_2(\delta^0\gamma)=-\frac 12\delta^0[\gamma,\gamma]$, i.e., we have the following diagram.
      \[
      \begin{tikzcd}[row sep=large]
         {\Omega^1_{im}(A;\frak k)^\Hor} & {\Omega^2_{im}(A;\frak k)^\Hor} \\
         {\Omega^1(M;\frak k)} & {\Omega^2(M;\frak k)}
         \arrow["{\mathbf c_2}", from=1-1, to=1-2]
         \arrow["{\delta^0}", from=2-1, to=1-1]
         \arrow["{\gamma\mapsto-\frac 12[\gamma,\gamma]}"', from=2-1, to=2-2]
         \arrow["{\delta^0}"', from=2-2, to=1-2]
      \end{tikzcd}
      \]
      Together with the left square in diagram \eqref{eq:weil_ideals}, this already shows that $F^\gamma$ is a curving of $(\C,v)+\delta^0\gamma$, proving the first part of the proposition. The remarkable fact that the 3-curvature does not change with this deformation, is a simple calculation which we provide here for completeness:
      \begin{align}
\begin{split}
  \label{eq:3_curvature_invariant}
  G^\gamma&=\d{}^{\nabla^\gamma}F^\gamma=(\d{}^\nabla+[\cdot,\gamma]_{\frak k})\big(F+\d{}^\nabla\gamma-\tfrac 12[\gamma,\gamma]_{\frak k}\big)\\
  &=G+R^\nabla\wedge\gamma-[\d{}^\nabla\gamma,\gamma]_{\frak k}+[F,\gamma]_{\frak k}+[\d{}^\nabla\gamma,\gamma]_{\frak k}-\tfrac 12[[\gamma,\gamma]_{\frak k},\gamma]_{\frak k}=G.
\end{split}
\end{align}
Here, we denoted by $\nabla^\gamma$ the connection on $\frak k$ associated to the IM connection $(\C,v)+\delta^0\gamma$, 
\begin{align}
\label{eq:nabla_gamma}
   \nabla^\gamma=\nabla+[\cdot,\gamma],
\end{align} 
noted that the following identity holds on forms of degree $k$,
\begin{align}
   \label{eq:d_nabla_gamma}
   \d{}^{\nabla^\gamma}=\d{}^\nabla+(-1)^k[\cdot,\gamma]_{\frak k},
\end{align}
and observed that the Jacobi identity for $[\cdot,\cdot]_{\frak k}$ implies $[[\gamma,\gamma]_{\frak k},\gamma]_{\frak k}=0$.
\end{proof}
On the other hand, there also exists another class of affine deformations of connection triples, which comes from non-uniqueness of curvings. Namely, given a connection triple $((\C,v),F)$, we may simply deform the curving $F$ by an invariant 2-form $\beta\in\Omega^2_{\ainv}(M;\frak k)$, and get a connection triple
\begin{align}
   \label{eq:affine_invariant}
   ((\C,v),F+\beta).
\end{align}
One difference with deformations from the previous proposition is that with these deformations, the 3-curvature also deforms: $G\ra G+\d{}^\nabla\beta$. This will be important when varying the functional.


By combining these two distinct types of affine deformations, we now obtain a decomposition of the set $\DD(A;\frak k)$ into affine spaces, determined by cohomological classes of IM connections. More precisely, let us fix a class $\chi$ defined by a primitive IM connection, that is,
\[
\chi\in \arc(A;\frak k)/\im\delta^0\subset H^1(W^{\bullet,1}(A;\frak k)),
\]
where $\arc(A;\frak k)\subset \A(A;\frak k)$ denotes the subset of primitive IM connections. The following lemma shows that the constrained subspace of connection triples 
\[
\DD_\chi(A;\frak k)=\set*{((\C,v),F)\in \A(A;\frak k)\times \Omega^2(M;\frak k)\given \delta^0 F=\Omega^{(\C,v)}, [\C,v]=\chi}
\]
is an affine space.
\begin{lemma}
\label{lemma:D_chi_affine}
  Let $\frak k$ be a bundle of ideals of a Lie algebroid $A$. For $\chi$ as above, $\DD_\chi(A;\frak k)$ is an affine space, modelled on the vector space $Q_\chi$ defined as the quotient
\[\begin{tikzcd}
	0 & {\Omega^1_\ainv(A;\frak k)} & {\Omega^1(M;\frak k)\oplus \Omega^2_\ainv(M;\frak k)} & {Q_\chi} & 0,
	\arrow[from=1-1, to=1-2]
	\arrow["j_\chi", from=1-2, to=1-3]
	\arrow[from=1-3, to=1-4]
	\arrow[from=1-4, to=1-5]
\end{tikzcd}\]
where $j_\chi(\gamma)=(\gamma,-\d{}^\nabla\gamma)$ and
$\nabla$ is the connection induced by an arbitrary representative of $\chi$.
\end{lemma}
\begin{proof}
  First, note that when restricted to the centre $z(\frak k)$, any linear connection $\nabla$ as above does not depend on the choice of a representative of $\chi$, by identity \eqref{eq:nabla_gamma}. This is important since invariant forms are centre-valued, so $\d{}^\nabla$ appearing in the statement is also independent of such a choice; moreover, $\d{}^\nabla$ preserves invariance of forms by diagram \eqref{eq:weil_ideals}, which altogether shows that the map $j_\chi$ is well-defined and only depends on the class $\chi$. The affine structure on $\DD_\chi(A;\frak k)$ is defined by
\begin{align}
  \label{eq:affine_DD_chi}
    ((\C,v),F)+\llbracket \gamma,\beta\rrbracket=((\C,v)+\delta^0 \gamma, F+\d{}^\nabla\gamma-\tfrac 12[\gamma,\gamma]+\beta)
\end{align}
  for any $\gamma\in\Omega^1(M;\frak k)$ and $\beta\in\Omega^2_\ainv(M;\frak k)$, where the double bracket $\llbracket \cdot,\cdot \rrbracket$ denotes the equivalence class of $(\gamma,\beta)$ in $Q_\chi$. That this is a well-defined affine structure on $\DD_\chi(A;\frak k)$ is a straightforward computation that follows from the definition of the quotient $Q_\chi$.
\end{proof}
\begin{remark}
  \label{rem:decompositions_affine}
  There is another way of looking at this lemma, which offers an explanation why the model space for $\DD_\chi(A;\frak k)$ is a quotient. Note that the set $\DD(A;\frak k)$ carries two equivalence relations:
  \begin{align*}
    ((\tilde\C,\tilde v),\tilde F)\sim_1 ((\C,v),F)&\iff (\tilde \C,\tilde v)-(\C,v)=\delta^0\gamma\text{ and }\tilde F=F^{\gamma}\text{ for some }\gamma \in\Omega^1(M;\frak k),\\
    ((\tilde\C,\tilde v),\tilde F)\sim_2 ((\C,v),F)&\iff (\tilde\C,\tilde v)=(\C,v)\text{ and }\delta^0(\tilde F- F)=0.
  \end{align*}
 The affine deformations \eqref{eq:affine_primitive} and \eqref{eq:affine_invariant} describe the equivalence classes defined by relations $\sim_1$ and $\sim_2$, respectively. The closure of the union of $\sim_1$ and $\sim_2$ yields another equivalence relation:
  \begin{align}
    ((\tilde\C,\tilde v),\tilde F)\sim_3 ((\C,v),F)
    &{\ \iff\ } (\tilde \C,\tilde v)-(\C,v)=\delta^0\gamma\text{ and }\tilde F=F^{\gamma}+\beta, \nonumber\\
    &\phantom{\ \iff\ }\text{for some }\gamma \in\Omega^1(M;\frak k), \beta\in\Omega^2_{\ainv}(M;\frak k),\nonumber\\
    &{\ \iff\ } [\tilde \C,\tilde v]=[\C,v]\in H^{1}(W^{\bullet,1}(A;\frak k)).\label{eq:tilde_3}
  \end{align}
  The quotient $Q_\chi$ appears since that the intersection of $\sim_1$ and $\sim_2$ is not trivial, and the equivalence classes of $\sim_3$ are precisely the affine spaces $\DD_\chi(A;\frak k)$.
\end{remark}
\subsection{Invariance of metric on the bundle of ideals}
Before diving into the proof of Theorem \ref{thm:main_theorem}, we fix some notation and make a few preliminary comments regarding the extra data needed for the construction of the variational problem. 

Since we are fixing a (pseudo) Riemannian metric and orientation on the base $M$ as well as a metric on $\frak k$, we get a pairing on $\frak k$-valued compactly supported differential forms on $M$:
\begin{align*}
  \innerr\cdot\cdot\colon \Omega_c^k(M;\frak k)\times \Omega_c^k
 (M;\frak k)\ra \R,\quad \innerr{\gamma_1}{\gamma_2}=\int_M\inner{\gamma_1}{\gamma_2}_{\frak k}\vol_M.
\end{align*}
This is a non-degenerate pairing since both metrics on $\frak k$ and $M$ are non-degenerate; the proof of this fact amounts to the fundamental lemma of calculus of variations. Given any linear connection $\nabla$ on $\frak k$, the \textit{formal adjoint} to $\d{}^\nabla$ with respect to $\innerr\cdot\cdot$, denoted in the Theorem \ref{thm:main_theorem} by $(\d{}^\nabla)^*$, is the unique map characterized by the equality 
$
\innerr{\d{}^\nabla\gamma_1}{\gamma_2}=\innerr{\gamma_1}{(\d{}^\nabla)^*\gamma_2},
$
for any forms $\gamma_1$ and $\gamma_2$ of degrees $k-1$ and $k$, respectively (see \cite{analysis_on_mflds}*{Proposition 1.2.11} for existence and uniqueness).

Similarly to the case of principal bundles, the metric on $\frak k$ is assumed to be invariant with respect to the adjoint action $A\curvearrowright\frak k$. The main reason for this assumption is gauge invariance (Theorem \ref{thm:gauge_invariance}).
\begin{definition}
   Given a bundle of ideals $\frak k$ of a Lie algebroid $A$, a metric $\inner\cdot\cdot_{\frak k}\in \Gamma(S^2\frak k^*)$ is said to be \textit{$\ad$-invariant}, if it satisfies the following identity, for any $\alpha\in\Gamma(A)$ and $\xi,\eta\in \Gamma (\frak k)$:
\begin{align}
   \label{eq:ad_invariance}
       \rho(\alpha)\inner\xi\eta_{\frak k}=\inner{[\alpha,\xi]}{\eta}_{\frak k}+\inner\xi{[\alpha,\eta]}_{\frak k}.
\end{align}
\end{definition}
A necessary condition for the existence of such a metric is that the typical fibre of $\frak k$ is a Lie algebra of compact type \cite{duistermaat}*{Theorem 3.6.2}. On the other hand, a sufficient condition is that $A$ is integrable by a proper groupoid $\G$. In this case, one first constructs an \textit{$\Ad$-invariant} metric, i.e., \[\inner\xi\eta_{\frak k}=\inner{\Ad_g\xi}{\Ad_g\eta}_{\frak k}, \quad\text{for all $g\in \G$ and $\xi,\eta\in\frak k_{s(g)}$}.\] This is done by means of a standard averaging argument using a Haar system and a cutoff function \cite{measures_on_stacks}, which is a groupoid  version of the construction of an $\Ad$-invariant inner product on the Lie algebra of a compact Lie group. Differentiation then shows that the obtained $\Ad$-invariant metric is also $\ad$-invariant.

Given an IM connection $(\C,v)$ for $\frak k\subset A$, $\ad$-invariance of the metric $\inner\cdot\cdot_{\frak k}$ is, however, not enough to guarantee that the induced linear connection $\nabla=\C|_{\Gamma(\frak k)}$ is a metric connection. This is now only guaranteed in the orbital directions,
\begin{align}
       \rho(\alpha)\inner\xi\eta_{\frak k}=\inner{\nabla_{\rho(\alpha)}\xi}{\eta}_{\frak k}+\inner\xi{\nabla_{\rho(\alpha)}\eta}_{\frak k},
\end{align}
which is not hard to see by combining equation \eqref{eq:ad_invariance} with condition \eqref{eq:c2} of an IM connection, or just looking at Definition \ref{def:gauge_triples} (iv). So, to ensure $\nabla$ is compatible with $\inner\cdot\cdot_{\frak k}$, one needs an additional assumption; we note that this condition is not needed in Theorems \ref{thm:main_theorem} and \ref{thm:gauge_invariance}.
\begin{definition}
   \label{def:transversal_metric_compatibility}
   An IM connection $(\C,v)$ for a bundle of ideals $\frak k\subset A$ is \emph{transversally compatible} with a given $\ad$-invariant metric $\inner{\cdot}{\cdot}_{\frak k}$, if every point $x\in M$ admits tangent vectors $(X_i)_i\subset T_xM$ ($i=1,\dots,\codim_x\F$) which induce a basis of the normal space $\nu_x(\F)=T_xM/\im\rho_x$, satisfying
   \begin{align}
   \label{eq:transversal_compatibility}
      X_i\inner{\xi}{\eta}_{\frak k}=\inner{\nabla_{X_i}\xi}{\eta}_{\frak k} + \inner{\xi}{\nabla_{X_i}\eta}_{\frak k}
   \end{align}
   for any $\xi,\eta\in\Gamma(\frak k)$.
In fact, this is a condition on the class $\chi=[\C,v]$, since ad-invariance and equation \eqref{eq:nabla_gamma} imply that if this holds for one representative IM connection in $\chi$, it holds for all.
\end{definition}

This property is favorable for the following reason. In the case when the connection $\nabla$ on $\frak k$ is compatible with $\inner\cdot\cdot_{\frak k}$, the formal adjoint $(\d{}^\nabla)^*$ has the following well-known expression on $\Omega_c^k(M;\frak k)$:
\begin{align}
\label{eq:delta_nabla}
(\d{}^\nabla)^*=(-1)^k\star^{-1}\!{\d{}^\nabla{\star}},
\end{align}
where $\star$ is the Hodge star operator with respect to the chosen metric and orientation on $M$. For completeness, we give a short proof. Define the map $\wedge\colon\Omega^k(M;\frak k)\times\Omega^l(M;\frak k)\ra \Omega^{k+l}(M)$ given on  simple tensors by $(\delta_1\otimes \xi_1)\wedge (\delta_2\otimes \xi_2)=\delta_1\wedge\delta_2 \inner{\xi_1}{\xi_2}_{\frak k}$. Metric compatibility implies
\[
\d(\gamma_1\wedge\gamma_2)=\d{}^\nabla\gamma_1\wedge \gamma_2+(-1)^{\deg(\gamma_1)}\gamma_1\wedge\d{}^\nabla\gamma_2,
\]
for $\frak k$-valued forms $\gamma_i$ of arbitrary degree. If $\gamma_1$ and $\gamma_2$ are of degrees $k-1$ and $k$ (resp.), taking $\star\gamma_2$ in place of $\gamma_2$, integrating $\d(\gamma_1\wedge {\star}\gamma_2)$ over $M$, and using Stokes' theorem, proves the claim:
\[
\innerr{\d{}^\nabla\gamma_1}{\gamma_2}=\int_M \d{}^\nabla\gamma_1\wedge \star \gamma_2=(-1)^k\int_M\gamma_1\wedge \d{}^\nabla{\star}\,\gamma_2=\innerr{\gamma_1}{(-1)^k\star^{-1}\!\d{}^\nabla{\star}\,\gamma_2}.
\]

\subsection{Euler--Lagrange equation and relation}
\label{sec:euler-lagrange}
We now arrive to the final piece of the proof of Theorem \ref{thm:main_theorem}, the derivation of the Euler--Lagrange equation \eqref{eq:YM.I} and relation \eqref{eq:YM.II}. We will prove it without the assumption of compactness on the base manifold $M$, by instead being explicit where compact support is required. In the notation from previous section, our action functional reads
\begin{align*}
   &\S\colon \DD_c(A;\frak k)\ra \R,\quad \S((\C,v),F)=\innerr{F}{F}+\mu \innerr{G}{G},\\
   &\text{where }\DD_c(A;\frak k)=\set*{((\C,v),F)\in\A(A;\frak k)\times \Omega^2_c(M;\frak k)\given 
      \delta^0F=\Omega^{(\C,v)}
    }.
\end{align*}
\begin{proof}[Proof of Theorem \ref{thm:main_theorem}]
First observe that a similar result as in Lemma \ref{lemma:D_chi_affine} holds for $\DD_c(A;\frak k)$, if we replace the spaces of forms on $M$ with compactly supported ones. Hence, to find the solutions of the constrained variational problem for the constraint $[\C,v]=const.$, we have to expand the expression
\begin{align*}
   \S\big(((\C,v),F)+\lambda\llbracket \gamma,\beta\rrbracket\big)&=\S((\C,v)+\lambda\delta^0\gamma,F^{\lambda\gamma}+\lambda\beta)\\
   &=\innerr{F^{\lambda\gamma}+\lambda\beta}{F^{\lambda\gamma}+\lambda\beta}+\mu\innerr{\d{}^{\nabla^{\lambda\gamma}}(F^{\lambda\gamma}+\lambda\beta)}{\d{}^{\nabla^{\lambda\gamma}}(F^{\lambda\gamma}+\lambda\beta)},
\end{align*}
for any $\gamma\in\Omega^1_c(M;\frak k)$ and $\beta\in\smash{\Omega^2_{\ainv,c}}(M;\frak k)$. Here, the terms read
\[
F^{\lambda\gamma}+\lambda\beta= F+\lambda(\d{}^\nabla\gamma+\beta)-\tfrac{\lambda^2}2[\gamma,\gamma]\quad\text{and}\quad \d{}^{\nabla^{\lambda\gamma}}(F^{\lambda\gamma}+\lambda\beta)=G+\lambda\d{}^\nabla\beta,
\]
where we already used that $\beta$ is centre-valued, and the second part of Lemma \ref{lemma:affine_primitive}. Hence, we obtain the following expansion in the parameter $\lambda$:
\begin{align}
   \label{eq:expansion_lambda_euler_lagrange}
\S\big(((\C,v),F)+\lambda\llbracket \gamma,\beta\rrbracket\big)&=\innerr FF+\mu\innerr GG\\
   &+\lambda\big(2\innerr{F}{\d{}^\nabla\gamma+\beta}+2\mu\innerr G{\d{}^\nabla \beta}\big)\nonumber\\
   &+\lambda^2(\innerr{\d{}^\nabla\gamma+\beta}{\d{}^\nabla\gamma+\beta}-2\innerr F{\tfrac 12[\gamma,\gamma]}+\mu \innerr{\d{}^\nabla\beta}{\d{}^\nabla\beta})\nonumber\\
   &+\O(\lambda^3)\nonumber.
\end{align}
We explicitly wrote out the quadratic term as well, for the later purpose of computing the Hessian and the tangent space to the space of solutions.
Looking at the linear term, we see that the triple $((\C,v),F)$ is critical if and only if for any $\gamma\in\Omega^1_c(M;\frak k)$ and $\beta\in\Omega^2_{\ainv,c}(M;\frak k)$, we have
\begin{align}
   0=\deriv\lambda 0\S\big(((\C,v),F)+\lambda\llbracket \gamma,\beta\rrbracket\big)=2(\innerr{(\d{}^\nabla)^* F}{\gamma}+\innerr{F+\mu(\d{}^\nabla)^* G}{\beta}).
\end{align}
Finally, this is zero for all $\gamma$ and $\beta$ if and only if the following PDE and PDR are fulfilled, 
\begin{align*}
   (\d{}^\nabla)^* F &=0,\\F+\mu\,(\d{}^\nabla)^* G &\perp \Omega^2_{\ainv,c}(M;\frak k).\tag*{\qedhere}
\end{align*}
\end{proof}
\begin{remark}
   The Euler--Lagrange conditions above make sense even when $F$ is not compactly supported, but in that case, the integral in the action \eqref{eq:ym_action} may not converge. In other words, even if a connection triple $((\C,v),F)$ fails to be in the domain of the action functional, its curvature pair can still satisfy the Euler--Lagrange conditions; for instance, see Example \ref{example:self_dual} (i).
\end{remark}

\subsection{Gauge invariance}
This section is devoted to the proof of Theorem \ref{thm:gauge_invariance}. Recall from Definition \ref{def:infinitesimal_gauge_transformations} that a \textit{gauge transformation} is a Lie algebroid automorphism $\Theta$ of $A\Ra M$ covering an orientation-preserving isometry $\vartheta$ of the base $M$, restricting on $\frak k\subset A$ to an isometry for the metric $\inner\cdot\cdot_{\frak k}$. The assumption $\Theta(\frak k)\subset \frak k$ on $\Theta$ is equivalent to being an automorphism of the short exact sequence \eqref{eq:ses_bdl_ideals}. 
\[\begin{tikzcd}[column sep=large, row sep=large]
	0 & {\frak k} & A & {B} & 0 \\
	0 & {\frak k} & A & {B} & 0
	\arrow[from=1-1, to=1-2]
	\arrow[hook, from=1-2, to=1-3]
	\arrow["{\Theta|_{\frak k}}", from=1-2, to=2-2]
	\arrow["\phi", from=1-3, to=1-4]
	\arrow["{\Theta}", from=1-3, to=2-3]
	\arrow[from=1-4, to=1-5]
	\arrow["{\Theta_B}", from=1-4, to=2-4]
	\arrow[from=2-1, to=2-2]
	\arrow[hook, from=2-2, to=2-3]
	\arrow["\phi", from=2-3, to=2-4]
	\arrow[from=2-4, to=2-5]
\end{tikzcd}\]
Here, $\Theta_B$ on $B=A/\frak k$ is defined in the obvious way, namely $\Theta_B[\alpha]=[\Theta(\alpha)].$ These three maps are algebroid morphisms covering a diffeomorphism $\vartheta\colon M\ra M$ on the base.

\begin{example} 
\label{ex:gauge_transformations}   
The most important classes of gauge transformations are the following ones, which automatically satisfy the requirement of being an isometry for any $\ad$-invariant metric $\inner\cdot\cdot_{\frak k}$ on $\frak k$. \
\begin{enumerate}[label={(\roman*)}]
   \item Flows of inner derivations $[\alpha,\cdot]$ of the Lie algebroid $A$, for sections $\alpha\in\Gamma(A)$ such that the flow of the (complete) vector field $\rho(\alpha)$ is an orientation-preserving isometry of $M$. In particular, any isotropy-valued section $\alpha\in\Gamma(\ker\rho)$ satisfies this extra condition since in this case the base map is just $\id_M$. We refer the reader to \eqref{eq:flow_of_derivation} for the definition of the flow of a derivation.
  \item If $A$ is integrable, we can consider algebroid automorphisms of the form $\Ad_b=(I_b)_*$, where $I_b$ is the inner automorphism by a bisection $b\in\mathrm{Bis}(\G)$ of an integrating $s$-connected Lie groupoid $\G$ of $A$, whose base map $\vartheta=s\circ b$ is an orientation-preserving isometry. In particular, this holds for all static bisections, i.e., those with $s\circ b=\id_M$. This relies on Lemma \ref{lem:pairing_A_invariant_implies_G_invariant} to ensure $\Ad_b$ is an isometry on $\frak k$; see the discussion after its proof for the definition of inner automorphisms $I_b$. This recovers the classical definition of gauge transformations on principal bundles, since they correspond precisely to static bisections of the gauge groupoid of a given principal bundle.
\end{enumerate}
\end{example}

To prove Theorem \ref{thm:gauge_invariance}, let us restate it in a more precise way. For this, we must first define the pullback of IM connections, and more generally, IM forms. Observe that on $\Omega^\bullet(M;\frak k)$ the pullback is defined on simple tensors as
$\Theta^*(\gamma\otimes\xi)=\vartheta^*\gamma\otimes (\Theta^{-1})_*\xi$. 
  The pullback of an IM form $(L,l)$ reads
\[
(\Theta^*(L,l)) (\alpha) = \big(\Theta^* (L(\Theta\alpha)), \Theta^* (l(\Theta\alpha))\big).
\]
In particular, we can pull back an IM connection $(\C,v)$ and again obtain an IM connection, denoted $(\C_\Theta,v_\Theta)=\Theta^*(\C,v)$, whose symbol is indeed a splitting since $\Theta(\frak k)\subset\frak k$.
The induced linear connection $\nabla^\Theta\coloneq \C_\Theta|_{\Gamma(\frak k)}$ is simply the pullback connection,
\begin{align*}
  \nabla^\Theta_X\xi=(\Theta^{-1})_*\nabla_{\vartheta_*X}(\Theta_*\xi),
\end{align*}
for any $X\in\vf(M)$ and $\xi\in\Gamma(\frak k)$. With this in mind, it is easy to see the following.
\begin{proposition}
  For a given Lie algebroid $A\Ra M$ with an $\ad$-invariant metric $\inner\cdot\cdot_{\frak k}$ on a bundle of ideals $\frak k$, transversal metric compatibility of IM connections is stable under gauge transformations.
\end{proposition}
\begin{proof}
  Recall any gauge transformation $\Theta$ restricts on $\frak k$ to an isometry. Hence, compatibility with a metric $\inner\cdot\cdot_{\frak k}$ is stable under gauge transformations, i.e., compatibility of $\nabla=\C|_{\Gamma(\frak k)}$ implies compatibility of $\nabla^\Theta= \C_\Theta|_{\Gamma(\frak k)}$. By $\ad$-invariance, so is transversal compatibility (Definition \ref{def:transversal_metric_compatibility}). 
\end{proof} 
\pagebreak

\begin{theorem}
   \label{thm:gauge_invariance_restated}
	Assume the setting of Theorem \ref{thm:main_theorem}.\ The generalized Yang--Mills action functional \eqref{eq:ym_action} is invariant under the pullbacks of infinitesimal gauge transformations $(\Theta, \vartheta)$, 
	\[
	\S(\Theta^*(\C,v),\Theta^* F)=\S((\C,v),F),
	\]
	for any triple $((\C,v),F)\in \DD_c(A;\frak k)$. In particular, the solution space of Euler--Lagrange equation and relation \eqref{eq:YM.I} and \eqref{eq:YM.II} is invariant under gauge transformations.
\end{theorem}
\begin{proof}
  First note that the curvature of the pullback IM connection is the pullback of its curvature, 
\begin{align}
   \label{eq:naturality_curvature}
     \Omega^{\Theta^*(\C,v)}=\Theta^*\Omega^{(\C,v)}.
\end{align}
  This is an immediate consequence of naturality of the horizontal exterior covariant derivative on the Weil complex. More precisely, a Lie algebroid automorphism $(\Theta,\vartheta)$ induces an automorphism of the Weil complex $W^{\bullet,k}(A;\frak k)$ for any fixed $k\geq 0$, by pulling back along $\Theta$, similarly as for IM forms. This pullback, also denoted $\Theta^*$, commutes with the horizontal exterior covariant derivative,
\begin{align}
   \label{eq:naturality_D}
    \D{}^{\Theta^*(\C,v)}\Theta^*=\Theta^*\D{}^{(\C,v)}.
\end{align}
  An argument for why this equality holds goes as follows. Recall that $\smash{\D{}^{(\C,v)}}$ is defined as $\D{}^{(\C,v)}=h^*\circ \d{}^\nabla$, and note $h^*$ and $\d{}^\nabla$ satisfy the following identities (in the notation of \cite{covariant_derivatives}):
  \begin{align}
   \label{eq:naturality_dh}
    \d{}^{\nabla^\Theta}\Theta^*=\Theta^*\d{}^\nabla\,\text{ and }\,h^*_\Theta\Theta^*=\Theta^*h^*,
  \end{align} 
  where $h_\Theta^*$ is the horizontal projection of Weil cochains with respect to the IM connection $\Theta^*(\C,v)$. The first identity here is clear, and the second is most easily seen in the $\vb$-algebroid picture, that is, it follows from $\d\Theta\circ \smash{h^{TA}_\Theta}=h^{TA}\circ\d\Theta$, where $h^{TA}_\Theta\colon TA\ra E_\Theta=\d\Theta^{-1}(E)$ is the horizontal projection with respect to the IM connection $\Theta^*(\C,v)$. Applying \eqref{eq:naturality_D} to $(\C,v)$ yields \eqref{eq:naturality_curvature}.
  
  Now, if $F$ is a curving of $(\C,v)$ with 3-curvature $G$, then $\Theta^*F$ is a curving of $\Theta^*(\C,v)$ with 3-curvature $\Theta^*G$. This follows from the already observed fact that $\Theta^*$ is an automorphism of the Weil cochain complex, and the left equation in \eqref{eq:naturality_dh}. Thus, both the curving and its 3-curvature satisfy
  \begin{align}
    \Theta^*F=(\Theta^{-1})_*\circ \vartheta^*F,\quad \Theta^*G=(\Theta^{-1})_*\circ \vartheta^*G,
  \end{align}
  where  the pullback $\vartheta^*\colon \Omega^\bullet(M;\frak k)\ra \Omega^\bullet(M;\vartheta^*\frak k)$ is defined on simple tensors by pulling back the form and leaving the coefficients alone. This concludes the proof of the theorem, since a gauge transformation preserves both metrics on $\frak k$ and $M$, as well as the orientation of $M$. 
\end{proof}
\begin{remark}
   The proof shows that gauge invariance is, in a certain sense, almost tautological, as it merely reflects the naturality of the structures involved in the framework. An alternative approach would be to work directly with the connection triples from Definition \ref{def:gauge_triples} through explicit computations, which would come at the expense of conceptual clarity.
\end{remark}

\subsection{Instantons and special cases}
We now inspect the equation \eqref{eq:YM.I} and relation \eqref{eq:YM.II} more closely, for some special cases. In this section only, we shall restrict our attention to the connection triples which have the property of transversal metric compatibility with the given ad-invariant metric $\inner\cdot\cdot_{\frak k}$ (Definition \ref{def:transversal_metric_compatibility}), to ensure that the formal adjoint has the favorable expression \eqref{eq:delta_nabla}.
\subsubsection{The case of vanishing invariant forms}
\label{sec:vanishing_invariant_forms}
Since invariant forms (of degree 2 and 3) play an important role in our framework, let us look at what happens when they are trivial:  $\Omega^k_\ainv(M;\frak k)=0$, for $k\geq 2$. For instance, this happens in the following two situations.
\begin{itemize}
   \item The leaves of the orbit foliation $\F$ have codimension at most 1, 
   \item The typical fibre $\frak g$ of $\frak k$ has a vanishing center; this is equivalent to semisimplicity of $\frak g$, since the existence of an $\ad$-invariant metric on $\frak k$ implies the Lie algebra $\frak g$ is of compact type. The curving $F$ is then uniquely determined by $\nabla$ with relation $R^\nabla=[-,F]_{\frak k}$ \cite{covariant_derivatives}*{Prop.\ 5.33}.
\end{itemize}
In this case, the 3-curvature vanishes, and the relation \eqref{eq:YM.II} is vacuously true, so the solutions to the variational problem are precisely the triples whose 2-curvature $F$ satisfies \eqref{eq:YM.I}. As a consequence, we have that $F$ is a harmonic 2-form with respect to the Hodge--Laplacian,
\begin{align}
   \label{eq:harmonic_F}
   \Delta F= \big(\!\d{}^\nabla(\d{}^\nabla)^*+(\d{}^\nabla)^*\d{}^\nabla\big) F=0.
\end{align}

Now, if additionally the base manifold $M$ is 4-dimensional, we can extend the usual notion of (anti) self-dual connections on principal bundles. Namely, a connection triple $(\sigma,\nabla,F)$ is called \textit{(anti) self-dual} if its 2-curvature satisfies
\begin{align}
   \label{eq:instantons_4d}
   F=\pm \star F.
\end{align}
Due to Bianchi identity and equation \eqref{eq:delta_nabla}, such an $F$ is automatically a solution of \eqref{eq:YM.I}. 
\begin{remark}
   \label{rem:self_consistency}
   For the equation \eqref{eq:instantons_4d} to be self-consistent, the index $s$ of the pseudo Riemannian metric on the base must be even (unless we work with the complexified bundle of ideals $\frak k$). This is due to the identity for the squared Hodge star operator, which holds on forms of degree $k$:
   \begin{align}
      \label{eq:starstar}
      \star\star=(-1)^{s+k(n-k)},
   \end{align}
   where $n=\dim M$ and $s$ is the index (the number of negative components) of the pseudo Riemannian metric on $M$. In particular, pseudo Riemannian manifolds with signature $(1,3)$, such as the Minkowski space, cannot admit real (anti) self-dual solutions, whereas they may exist on Riemannian manifolds.
\end{remark}

\subsubsection{Instantons in 5D}
As already mentioned in the introduction, the generalization to the non-transitive setting allows us to introduce the class of  \textit{(anti) self-dual} connection triples over a 5-dimensional base manifold $M$. These are the triples $(\sigma,\nabla,F)$ whose  curvature pair $(F,G)$ satisfies
\begin{align}
   \label{eq:instantons_5d}
   G=\pm\star F.
\end{align}
Note that $\d{}^\nabla G=R^\nabla\wedge F=[F,F]_{\frak k}=0$ always holds, so \eqref{eq:instantons_5d} implies \eqref{eq:YM.I}, and moreover, $\nabla$ is necessarily flat, since \eqref{eq:instantons_5d} implies $F$ is centre-valued. The curvature pair further satisfies \eqref{eq:YM.II} if we choose $\mu=(-1)^s$, where $\mu$ is the constant from Theorem \ref{thm:main_theorem}.\ Indeed, equations \eqref{eq:delta_nabla} and \eqref{eq:instantons_5d} imply that the curvature pair satisfies the PDE 
\begin{align}
   \label{eq:YM.IIe}
   F+\mu\, (\d{}^\nabla)^* G =0,
\end{align}
which is clearly stronger than the PDR \eqref{eq:YM.II}. In contrast with harmonicity that we saw in equation \eqref{eq:harmonic_F} for the 4-dimensional case, we now have
\[
\Delta F= \big(\!\d{}^\nabla(\d{}^\nabla)^*+(\d{}^\nabla)^*\!\d{}^\nabla\big) F=-\tfrac 1\mu F,
\]
that is, $F$ is an eigenform of the Hodge--Laplacian induced by $\nabla$, with eigenvalue $-\frac 1\mu=(-1)^{s+1}$. 

\begin{remark}
   \label{rem:riemannian_contradiction}
   We observe that in the Riemannian case, the Hodge--Laplace operator $\Delta$ can only have non-negative eigenvalues. Indeed, if $\psi\in\Omega^\bullet_c(M;\frak k)$ is its eigenform with eigenvalue $\lambda$, then
     \[
    \lambda\innerr{\psi}{\psi}=\innerr{\Delta \psi}{\psi}=\innerr{\d{}^\nabla\psi}{\d{}^\nabla\psi}+\innerr{(\d{}^\nabla)^*\psi}{(\d{}^\nabla)^*\psi}\geq 0.
  \]
  Therefore, the structure constant must be negative, $\mu< 0$, as a necessary condition for the equation \eqref{eq:YM.IIe} to admit a solution. However, for \eqref{eq:instantons_5d} to imply \eqref{eq:YM.IIe} (and thus \eqref{eq:YM.II}), we must have $\mu=(-1)^s$, which now brings us to a reversed situation compared to Remark \ref{rem:self_consistency}: the Riemannian case does not allow for self-dual connection triples whose 2-curvature is compactly supported (unless we complexify our setting), and the pseudo-Riemannian case in general does. In the following simple example, (i) shows that this restriction does not apply to forms that are not compactly supported, and (ii) shows this restriction indeed applies only to the Riemannian case.
\end{remark}
\begin{example}[Totally intransitive case]
   \label{example:self_dual}
   Suppose $V=\R_M$ is the trivial line bundle over a manifold $M$, with trivial Lie bracket and anchor, and take the whole $\frak k=V$ as the bundle of ideals. We have $\Omega_\ainv^\bullet (M;\frak k)=\Omega^\bullet(M)$, and the connection triples are just $(\sigma=0, \nabla,F)$, where $F\in\Omega^2(M)$ and $\nabla$ is a flat connection on $V$ (with no additional compatibility conditions). In the following two examples of self-dual connection triples on $V$ in 5 dimensions, we let $\nabla$ be the canonical flat connection.
   \begin{enumerate}
   \item Let $M=\R^5$ be equipped with the Euclidean metric and coordinates $(x_0,\dots, x_4)$. We define
   \[
   F=\mathrm e^{\pm x_4}(\d x_0\wedge \d x_1+\d x_2\wedge \d x_3),
   \]
   and now it is easy to see there holds $\d F=\pm\star F$, so $(\sigma=0,\nabla, F)$ defines an (anti)  self-dual connection triple on $V$.
   \item For a compact pseudo-Riemannian example, let $M=(\R/2\pi\Z)^5$ be the flat Lorentzian 5-torus equipped with coordinates $(x_0,\dots,x_4)$ and metric $g=-\d x_0^2+\sum_{i\geq 1}\d x_i^2$. Define 
   \[
   F= \cos(x_4)\alpha -\sin(x_4)\beta,
   \]
   where $\alpha=\d x_0\wedge \d x_1$ and $\beta = \d x_2\wedge \d x_3$. We have $\star \alpha= -\d x_4\wedge \beta$ and $\star \beta =\d x_4\wedge \alpha$, hence 
   \[\star F=-\cos x_4\d x_4\wedge \beta - \sin x_4\d x_4\wedge \alpha = \d F.\]
   Therefore, $(\sigma=0,\nabla,F)$ defines a self-dual connection triple. We remark that it is more naturally expressed with complex forms. Namely, letting $\omega=\alpha+i\beta$, we have $\star \omega = i \d x_4\wedge \omega$, and the form $F_{\mathbb C}=\mathrm e^{i x_4}\omega\in\Omega^2(M;\mathbb C)$ satisfies $\d F_{\mathbb C}=\star F_{\mathbb C}$; the form $F$ is its real part.
   \end{enumerate}
\end{example}
   \begin{remark}
      \noindent This example, as well as equation \eqref{eq:starstar} and Remark \ref{rem:riemannian_contradiction}, indicates it should be interesting to consider \textit{complex} Lie algebroids \cite{complex_algebroids}. An interesting example of a complex Lie algebroid already comes from complexifying a real Lie algebroid (\textit{loc.\,cit.}\ \sec 2).  
      In this context, introducing \textit{complex} connection triples should allow for a self-duality relation $G=\pm i\star F$. Indeed, given any constant $\mu\in\R$, for \eqref{eq:YM.IIe} to hold,
      \[
      G=\pm\tfrac 1{\sqrt{(-1)^s\mu}}\star F,
      \]
      must be fulfilled, hence $\sgn\mu=(-1)^s$ is no longer a requirement if we work in the complex setting. We will, however, not explore this direction further in this paper.
   \end{remark}

\subsection{Tangent space to the space of solutions}
We now obtain the (formal) tangent space to the space of solutions of the constrained variational problem \eqref{eq:ym_action}. For simplicity of notation, we shall work with a compact base manifold $M$, but we note that the tangent space can also be obtained for the non-compact case by restricting to compactly supported forms. As usual, we identify the tangent space of the affine space $\DD_\chi(A;\frak k)$, at any connection triple $(\sigma,\nabla,F)$, with the vector space on which it is modelled, i.e., $Q_\chi$. 

\begin{proposition}
  \label{prop:formal_tangent_space}
  Assume the setting of Theorem \ref{thm:main_theorem}, and let $(\sigma,\nabla,F)\in\DD_\chi(A;\frak k)$ be a solution to the constrained variational problem \eqref{eq:ym_action}. The Hessian at $(\sigma,\nabla,F)$ is given by
  \begin{align}
\begin{split}
      \label{eq:hessian}
     H_{(\sigma,\nabla,F)}\llbracket \gamma,\beta\rrbracket&=\innerr{(\d{}^\nabla)^*(\d{}^\nabla\gamma+\beta)-\widehat F(\gamma)}{\gamma}\\
     &+\innerr{(\d{}^\nabla\gamma+\beta)+\mu(\d{}^\nabla)^*\!\d{}^\nabla\beta}{\beta},
\end{split}
\end{align}
where $\widehat F\colon \Omega^1(M;\frak k)\ra \Omega^1(M;\frak k)$ is given by $\widehat F(\gamma)=\star^{-1}[\gamma,\star F]$.  
The formal tangent space at $(\sigma,\nabla,F)$ to the space  of all solutions in $\DD_\chi(A;\frak k)$ is given by all $\llbracket \gamma,\beta \rrbracket\in Q_\chi$ which satisfy 
  \begin{gather}
    (\d{}^\nabla)^*(\d{}^\nabla\gamma+\beta)=\widehat F(\gamma),\label{eq:fts_1}\\
    (\d{}^\nabla\gamma+\beta)+\mu(\d{}^\nabla)^*\!\d{}^\nabla\beta\perp \Omega^2_\ainv(M;\frak k).\label{eq:fts_2}
  \end{gather}
\end{proposition}
\begin{proof}
We read out the Hessian from the second-order coefficient in expansion \eqref{eq:expansion_lambda_euler_lagrange}:
\begin{align*}
   H_{(\sigma,\nabla,F)}\llbracket\gamma,\beta\rrbracket&=\frac 12\frac{d^2}{d\lambda^2}\Big|_{\lambda=0}\S((\sigma,\nabla,F)+\lambda\llbracket\gamma,\beta\rrbracket)\\
   &=\innerr{\d{}^\nabla\gamma+\beta}{\d{}^\nabla\gamma+\beta}-\innerr F{[\gamma,\gamma]}+\mu \innerr{\d{}^\nabla\beta}{\d{}^\nabla\beta}.
\end{align*}
To bring it into the desired form, note that for any $\gamma_1,\gamma_2\in\Omega^1(M;\frak k)$, there holds
\[
\innerr{\left[\gamma_1,\gamma_2\right]}{F}=\int_M [\gamma_1,\gamma_2]\wedge\star F=\int_M\gamma_1\wedge[\gamma_2,\star F]=\innerr{\gamma_1}{\star^{-1}[\gamma_2,\star F]}=\innerr{\gamma_1}{\widehat F(\gamma_2)},
\]
where we used in the second equality that $\ad$-invariance of $\inner\cdot\cdot_{\frak k}$ implies
\[
[\alpha,\beta]\wedge\gamma=\alpha\wedge[\beta,\gamma],
\]
for any forms $\alpha,\beta,\gamma\in\Omega^\bullet(M;\frak k)$. As an aside, note that the map $\widehat F$ is clearly self-adjoint. Equality \eqref{eq:hessian} now follows, and identities \eqref{eq:fts_1} and \eqref{eq:fts_2} are then a consequence of the general fact that the tangent space to the subspace of critical points of a function is given by the kernel of its Hessian; in our case,  $\ker H_{(\sigma,\nabla,F)}$. We hereby provide an alternative proof, which assumes $\nabla$ is compatible with $\inner\cdot\cdot_{\frak k}$. We want to differentiate \eqref{eq:YM.I}, so by varying $(\sigma,\nabla,F)\rightarrow (\sigma,\nabla,F)+\lambda\llbracket\gamma,\beta\rrbracket$, we get
\begin{align*}
(\d{}^{\nabla^{\lambda\gamma}})^*(F^{\lambda\gamma}+\lambda\beta)&= \star^{-1}(\d{}^\nabla+(-1)^{\dim M-2}\lambda[\cdot,\gamma]\big)\star(F^{\lambda\gamma}+\lambda\beta)\\
&=(\d{}^{\nabla})^*F+\lambda\big(-{\star^{-1}}[\gamma,\star\, F]+(\d{}^\nabla)^* \d{}^\nabla\gamma+(\d{}^\nabla)^*\beta\big).
\end{align*}
where we used equation \eqref{eq:d_nabla_gamma}. Setting this expression to zero and differentiating at $\lambda=0$, we obtain the equation \eqref{eq:fts_1}. 
To differentiate the second identity \eqref{eq:YM.II}, we similarly compute
\begin{align*}
  (F^{\lambda\gamma}+\lambda\beta)+\mu(\d{}^{\nabla^{\lambda\gamma}})^*(G^{\lambda\gamma}+\lambda\d{}^{\nabla^{\lambda\gamma}}\beta)&=(F^{\lambda\gamma}+\lambda\beta)-\mu\star^{-1}(\d{}^\nabla+
\lambda\cancel{[\cdot,\gamma]})\star(G+\lambda\d{}^\nabla\beta),
\end{align*}
where we used Lemma \ref{lemma:affine_primitive} and the fact that invariant forms are centre-valued. Setting this expression orthogonal to $\Omega^2_\ainv(M;\frak k)$, and differentiating at $\lambda=0$, we obtain  \eqref{eq:fts_2}.
\end{proof}
\begin{remark}
   \label{rem:underdetermined}
   From the second part of the proof, we can also see precisely in what way the equation and relation \eqref{eq:YM.I} and \eqref{eq:YM.II} are underdetermined: if we deform a solution as $(\sigma,\nabla,F)\rightarrow (\sigma,\nabla,F)+\llbracket\gamma,\beta\rrbracket$, it remains a solution if and only if $\gamma$ and $\beta$ satisfy \eqref{eq:fts_1} and \eqref{eq:fts_2}.
\end{remark}

\subsection{Yang--Mills theory on groupoids and relation with the van Est map}
\label{sec:global_ym}
We now provide the global analogue of the framework introduced for Lie algebroids, and establish a relationship between the global and infinitesimal Yang--Mills theory. The global setting will also enable us to describe Yang--Mills theory for bundle gerbes.

To begin with, given a bundle of ideals $\frak k$ of a groupoid $\G$ (Definition \ref{def:boi}), the metric $\inner\cdot\cdot_{\frak k}$ on $\frak k$ is now assumed to be $\Ad$-invariant, that is, 
	\begin{align*}
		\inner{\Ad_g\xi}{\Ad_g\eta}_{\frak k}=\inner{\xi}{\eta}_{\frak k}.
	\end{align*}
for any $g\in\G$ and $\xi,\eta\in \frak k_{s(g)}$. This condition implies $\ad$-invariance, and the converse holds if $\G$ is $s$-connected (Lemma \ref{lem:pairing_A_invariant_implies_G_invariant}). Moving on, let us define the global counterparts of connection triples.
\begin{definition}
   Given a bundle of ideals $\frak k$ on a Lie groupoid $\G\rra M$, a MEC $\omega\in\A(\G;\frak k)$ is said to be \textit{primitive}, if its curvature is multiplicatively exact, i.e., $[\Omega^\omega]=0$. In other words,
   \[
   \Omega^\omega =\delta^0_\G F,
   \]
   for some $F\in\Omega^2(M;\frak k)$, where we denote  by $\delta^0_\G$ the differential of the Bott--Shulman--Stasheff complex at level zero, see equation \eqref{eq:invariant}. Such a 2-form $F$ is called a \textit{curving} of MEC $\omega$, and the subset of all primitive MECs is denoted by $\arc(\G;\frak k)\subset \A(\G;\frak k)$.
\end{definition}
Analogously to the infinitesimal case, the domain of the action is now defined by:
\begin{align}
   \label{eq:action_domain_G}
   \DD(\G;\frak k)=\set*{(\omega,F)\in\A(\G;\frak k)\times \Omega^2(M;\frak k)\given 
   \Omega^\omega=\delta^0_\G F
   },
\end{align}
where we again emphasize that we need to restrict to curvings $F$ with compact support, if we do not assume the base manifold $M$ is compact. We now arrive to the groupoid version of Theorem \ref{thm:main_theorem}.

\begin{theorem}
   \label{thm:main_theorem_groupoids}
      Let $\frak k$ be a bundle of ideals of a Lie groupoid $\G$ over a compact, oriented, pseudo-Riemannian manifold $M$, and let $\inner\cdot\cdot_{\frak k}$ be an $\Ad$-invariant bundle metric on $\frak k$. For any pair $(\omega,F)$ consisting of a MEC and its curving, the following statements are equivalent.
      \begin{enumerate}[label={(\roman*)}]
         \item The pair $(\omega,F)$ is a solution to the constrained variational problem for the action functional 
   \begin{align}
      \label{eq:ym_action_G}
       \S_\G\colon\DD(\G;\frak k)\ra \R,\quad \S_\G(\omega,F)=\int_M \inner FF_{\frak k}\vol_M+\mu\int_M \inner GG_{\frak k}\vol_M,\  (\mu\in\R),
   \end{align}
      subject to the constraint that the class $[\omega]$ is constant in the cohomology of the Bott--Shulman--Stasheff complex \eqref{eq:bss} of the Lie groupoid $\G$.
      \item The curvature pair $(F,G=\d{}^\nabla F)$ satisfies the following Euler--Lagrange equation and relation,
      \begin{align}
            (\d{}^\nabla)^* F&=0,\label{eq:YM.Ig}\\
            F+\mu\,(\d{}^\nabla)^* G &\perp \Omega^2_{\ginv}(M;\frak k),\label{eq:YM.IIg}
      \end{align}
      where $\nabla$ is the linear connection on $\frak k$ induced by $\omega$, see e.g., \cite{covariant_derivatives}*{\textit{Proposition 3.9}}, and we denoted by $\Omega_{\ginv}^\bullet(M;\frak k)=\ker(\delta^0_\G)$ the space of $\G$-invariant forms on the base.
      \end{enumerate}
\end{theorem}
\begin{proof}
   The proof of this theorem is almost identical to the infinitesimal case, the only difference being in the use of the map $\delta^0_\G$ in place of $\delta^0_A\colon\Omega^k(M;\frak k)\ra \Omega^k_{im}(A;\frak k)$. Let us be more precise. For a fixed cohomological class $\chi\in\arc(\G;\frak k)/\im\delta^0_\G$, similarly as in \sec\ref{sec:affine_deformations}, the space of constrained pairs
\[
\DD_\chi(\G;\frak k)=\set*{(\omega,F)\in \A(\G;\frak k)\times \Omega^2(M;\frak k)\given \delta^0 F=\Omega^{\omega}, [\omega]=\chi}
\]
is an affine space, modelled on the following vector space (compare with Lemma \ref{lemma:D_chi_affine}):
\[
\big(\Omega^1(M;\frak k)\oplus \Omega^2_{\ginv}(M;\frak k)\big)/\Omega^1_{\ginv}(M;\frak k).
\]
Now, the rest of the proof of the theorem above consists of the same computation as in \sec\ref{sec:euler-lagrange}.
\end{proof}
 In what follows, we will see that the global and infinitesimal framework are equivalent, up to connectedness conditions on $\G$.

\subsubsection{Relation between the global and infinitesimal Yang--Mills theory}

Let $A\Ra M$ now be the Lie algebroid of a Lie groupoid $\G\rra M$. The Bott--Shulman--Stasheff complex \eqref{eq:bss} and the Weil complex \eqref{eq:weil_complex} are related by the so-called \textit{van Est map}; this is a cochain map, which we will denote
\[
\ve \colon \Omega^k(\G^{(\bullet)};\frak k)\ra W^{\bullet,k}(A;\frak k).
\]
See, for instance, \cites{diff_cohomology, weil, homogeneous,simplicial_van_est} for the definition and the proof that it commutes with the differentials. For our purpose, we only need its description on 1-cocycles: given any multiplicative form $\omega\in\Omega_m^k(\G;\frak k)$, the corresponding IM form $(L,l)=\ve(\omega)\in \Omega^k_{im}(A;\frak k)$ reads
\begin{align}
   \label{eq:ve_formula}
\begin{split}
    L(\alpha)_x(X_i)_{i=1}^k=\deriv\lambda 0\Ad_{\phi^{\alpha^L}_\lambda(1_x)}\, \omega\big({\d(\phi^{\cev\alpha}_\lambda)}_{1_{x}} (X_i)\big)_i,\quad l(\alpha)=u^*(\iota_{\alpha^L}\omega).
\end{split}
\end{align}
Since $\ve$ is a cochain map, it indeed maps multiplicative forms on $\G$ to IM forms on $A$, and this restriction will again be denoted $\ve\colon \Omega_m^k(\G;\frak k)\ra \Omega^k_{im}(A;\frak k)$. It satisfies the following important properties that we will need to relate the global and infinitesimal Yang--Mills theory.
\begin{itemize}
   \item If $\G$ is $s$-connected, the $\G$-invariant and $A$-invariant forms on the base manifold $M$ coincide, $\Omega^\bullet_\ginv(M;\frak k)=\Omega^\bullet_\ainv(M;\frak k)$, and the restriction $\ve\colon \Omega_m^\bullet(\G;\frak k)\ra \Omega^\bullet_{im}(A;\frak k)$ is injective.
   \item If $\G$ is $s$-simply connected, the restriction $\ve\colon \Omega_m^\bullet(\G;\frak k)\ra \Omega^\bullet_{im}(A;\frak k)$ is an isomorphism.
\end{itemize}
Focusing on connections, we see that the map $\ve$ restricts to a map between MECs on $\G$ and IM connections on $A$:  the defining condition $\omega|_{\frak k}=\id_{\frak k}$ on a MEC $\omega$ translates via $\ve$ to the condition that the symbol of the IM form $\ve(\omega)$ restricts on $\frak k$ to $\id_{\frak k}$. The following proposition shows that $\ve$ induces a map between the dynamical fields in the global and the infinitesimal setting, and up to connectedness conditions on $\G$, the variational problems of the two settings are the same.
\begin{proposition}
   \label{prop:global_vs_inf}
   Let $\frak k$ be a bundle of ideals of a Lie groupoid $\G$, with Lie algebroid $A$, over a compact, oriented, pseudo-Riemannian manifold $M$, and let $\inner\cdot\cdot_{\frak k}$ be an $\Ad$-invariant bundle metric on $\frak k$. The map 
   $
   \overline{\ve}\colon \DD(\G;\frak k)\ra \DD(A;\frak k)$, $\overline{\ve}(\omega,F)=(\ve(\omega),F)
   $
   has the following properties.
   \begin{enumerate}
      \item If $(\ve(\omega),F)$ is a solution to the variational problem \eqref{eq:ym_action} in the infinitesimal setting, then  $(\omega,F)$ is a solution to the variational problem \eqref{eq:ym_action_G} in the global setting. In other words,
      \[
      \overline{\ve}{}^{-1}(\set{\text{solution space on $A$}})\subset \set{\text{solution space on $\G$}}.
      \]
      \item If $\G$ is $s$-connected, the implication in (i) is an equivalence, and the inclusion is an equality.
      \item If $\G$ is $s$-simply connected, $\overline{\ve}$ is a bijection which restricts to a bijection between the solution spaces of the respective variational problems on $\G$ and $A$.
   \end{enumerate}
\end{proposition}
   \noindent Diagrammatically, the situation of the last proposition is depicted as follows.
\[\begin{tikzcd}
	{(\ve(\omega),F)\text{ is a solution for }\eqref{eq:ym_action}} & {(\omega,F)\text{ is a solution for }\eqref{eq:ym_action_G}}
	\arrow[Rightarrow, from=1-1, to=1-2]
	\arrow["{\text{if }\G\text{ is $s$-connected}}", bend left=12, Rightarrow, from=1-2, to=1-1]
\end{tikzcd}\]
\begin{proof}
First note that the map $\overline \ve$ is well-defined since 
\[
\Omega^{\ve(\omega)}=\ve(\Omega^\omega)=\ve(\delta^0_\G F)=\delta^0_AF,
\]
where we used \cite{covariant_derivatives}*{Theorem 4.14} in the first equality, and the equality 
\begin{align}
   \label{eq:ve_delta}
   \ve\circ\delta_\G^0=\delta^0_A 
\end{align}
in the third equality, which comes from the fact that $\ve$ is a cochain map. The point (i) follows from Theorems \ref{thm:main_theorem} and \ref{thm:main_theorem_groupoids}, together with the fact that any $\G$-invariant form is also $A$-invariant, that is, $\Omega^\bullet_\ginv(M;\frak k)\subset \Omega^\bullet_\ainv(M;\frak k)$, which also comes from \eqref{eq:ve_delta}. The points (ii) and (iii) follow from the properties of the van Est map $\ve$ listed before the proposition.
\end{proof}

We conclude this section with a brief note on gauge invariance. Namely, a global version of Theorem \ref{thm:gauge_invariance_restated} is also possible, and the reformulation is straightforward: in the setting of Lie groupoids, a (global) \textit{gauge transformation} is defined as a Lie groupoid automorphism $\Theta\colon\G\ra \G$ covering an orientation-preserving isometry on the base $M$, whose induced Lie algebroid map $\Theta_*=\d\Theta|_A\colon A\ra A$ restricts on $\frak k$ to an isometry for the metric $\inner\cdot\cdot_{\frak k}$. As before, an important class of examples of gauge transformations comes from  inner automorphisms induced by (static) global bisections, as defined in 
\eqref{eq:inner_automorphism}. Moreover, an analogous naturality argument as in the infinitesimal case shows the following.
\begin{theorem}
   Assume the setting of Theorem \ref{thm:main_theorem_groupoids}. The action functional \eqref{eq:ym_action_G} is invariant under the pullbacks by global gauge transformations. In particular, they preserve the solution space to the Euler--Lagrange equation \eqref{eq:YM.Ig} and relation \eqref{eq:YM.IIg}.
\end{theorem}
 
Finally, recall that Lie's second theorem holds for Lie groupoids \cite{lie2}. Therefore, if $\G$ is $s$-simply connected (in other words, if $\G$ is isomorphic to the Weinstein groupoid of $A$), then the infinitesimal gauge transformations are in bijection with the global gauge transformations. Together with Proposition \ref{prop:global_vs_inf}, this means that in the case when $\G$ is $s$-simply connected, there is no difference between the global and infinitesimal Yang--Mills theory.

\needspace{2cm}
\section{Examples}
A first example of a solution to the Yang--Mills equation and relation \eqref{eq:YM.I} and \eqref{eq:YM.II} has already been given in Example \ref{example:self_dual}, pertaining to the simple case of a totally intransitive algebroid. We now produce more examples, of which bundle gerbes (\sec\ref{sec:ym_gerbes}) are the most important.
\subsection{Almeida--Molino algebroid (transitive, non-integrable case)}
\label{ex:presymplectic}
Let us look at the well-known transitive Lie algebroid of Almeida and Molino \cites{molino, integrability, lectures_on_integrability, integration_transitive}. Historically, this was the first known example of a non-integrable Lie algebroid. To construct it, fix a closed 2-form $\omega\in\Omega^2(M)$ on a manifold $M$, and consider the Whitney sum \[A_\omega=TM\oplus \R_M,\] 
where $\R_M$ denotes the trivial line bundle over $M$. The anchor of $A_\omega$ is given by $\rho=\pr_{TM}$ and the bracket is defined on the sections as
\begin{align*}
[(X,f),(Y,g)]_\omega=([X,Y],Xg-Yf+\omega(X,Y)),
\end{align*}
for any $X,Y\in\vf(M)$ and $f,g\in \Gamma(\R_M)=C^\infty(M)$. The associated short exact sequence reads
\begin{align}
   \label{eq:ses_almeida_molino}
   \begin{tikzcd}[ampersand replacement=\&]
   	0 \& {\R_M} \& A_\omega \& TM \& 0.
   	\arrow[from=1-1, to=1-2]
   	\arrow[from=1-2, to=1-3]
   	\arrow["\rho", from=1-3, to=1-4]
   	\arrow[from=1-4, to=1-5]
   \end{tikzcd}
\end{align}
Any splitting $\sigma$ of this sequence must be of the form
\begin{align*}
\sigma(X)=(X,\theta(X)),
\end{align*}
for some $\theta\in\Omega^1(M)$, so we can identify the space of splittings with $\Omega^1(M)$. Due to transitivity, connection triples are determined by the splitting, so $\DD(A_\omega)\cong\Omega^1(M)$. More precisely, the induced connection $\nabla$ on $\R_M$ is the canonical flat connection,
\begin{align*}
\nabla_X f=[(X,\theta(X)),(0,f)]_\omega=Xf,
\end{align*}
and the curvature of a splitting $\theta\in\Omega^1(M)$ is identified with a form $F^\theta\in\Omega^2(M)$ by computing:
\begin{align*}
F^\sigma(X,Y)&=([X,Y],\theta[X,Y])-([X,Y],X(\theta(Y))-Y(\theta(X))+\omega(X,Y))\\
&=(0,-\d\theta(X,Y)+\omega(X,Y)).
\end{align*}
Hence, $F^\theta=\omega-\d\theta$, so we observe that the curvature of a splitting $\theta$ vanishes if and only if $\omega$ is exact with the form $\theta$ as its primitive.

In order to talk about Yang--Mills theory on $A_\omega$, we also need to fix a metric on $M$ and a fibrewise metric on $\R_M$. It is easy to see that the usual Euclidean inner product on $\R_M$ (i.e., the product of functions) is ad-invariant, which in this case reads out as the product rule $X(fg)=(Xf)g+f(Xg)$. Thus, the action reads
\begin{align}
\label{eq:functional_almeida_molino}
\S\colon\Omega^1(M)\rightarrow \R,\quad \S(\theta)=\int_M\inner{\omega-\d\theta}{\omega-\d\theta}\vol_M.
\end{align}
By Theorem \ref{thm:main_theorem}, a splitting $\theta\in\Omega^1(M)$ is critical for this action if and only if
\begin{align}
\label{eq:ym_example_abelian}
\d{}\star(\omega-\d\theta)=0,
\end{align}
where $\star$ is the Hodge-$\star$ operator associated to the given metric on $M$. Assuming $M$ is compact and Riemannian, we now observe that by Hodge theorem \cite{foundations_manifolds}*{Theorem 6.1}, critical splittings exist. Indeed, since $\omega$ is closed, it defines a cohomology class $[\omega]\in H_{dR}^2(M)$, which in turn admits a unique harmonic representative $\eta\in\Omega^2(M)$, so  $\eta=\omega-\d\theta$ for some $\theta\in\Omega^1(M)$. Furthermore, uniqueness of $\eta$ implies that $\theta$ is only unique up to a closed 1-form; this is just a rephrasing of the fact that the PDE \eqref{eq:ym_example_abelian} is underdetermined, as observed in Remark \ref{rem:underdetermined}. More precisely, for a critical splitting $\theta$, the splitting $\theta+\tau$ is again critical if and only if $\d\tau=0$, so all critical splittings have the same curvature, $\eta$. At last, by rescaling the action functional \eqref{eq:functional_almeida_molino}, we can write it as a difference of a kinetic and a potential term (i.e., a Lagrangian), and conclude the following.
\begin{corollary}
   The harmonic representative of a closed 2-form $\omega\in\Omega^2(M)$ on a compact oriented Riemannian manifold $M$ is given by the curvature $F^\theta$ of a critical splitting $\theta\in\DD(A_\omega)\cong\Omega^1(M)$ of the  Almeida--Molino algebroid $A_\omega$, for the action functional
   \[
\S(\theta)=\int_M\big(\tfrac 12\inner{\d\theta}{\d\theta}-\inner{\omega}{\d\theta}\big)\vol_M.
\]
\end{corollary}
Observe that one can also exchange $\R$ for an arbitrary abelian Lie algebra, and the essential features of the resulting example will remain the same as above.

\subsection{Bundle gerbes (non-transitive, integrable case)}
\label{sec:ym_gerbes}
In this section, we show the global framework developed in \sec\ref{sec:global_ym}  yields a Yang--Mills theory for bundle gerbes over manifolds. For simplicity, we will focus only on $S^1$-bundle gerbes over manifolds, which are defined as central $S^1$-extensions of submersion groupoids. Let us recall these notions.

\begin{definition}
   An \emph{$S^1$-extension} of a submersion groupoid is a Lie groupoid $\G\rra M$, together with a surjective submersive Lie groupoid morphism $\Phi\colon \G\ra M\times_\pi M$ onto the submersion groupoid,
   \[
   M\times_\pi M=\set{(y,x)\in M\times M\given \pi(y)=\pi(x)},
   \]
   of a surjective submersion $\pi\colon M\ra N$, such that the kernel of $\Phi$ is the trivial abelian Lie group bundle $S^1_M=M\times S^1$. The situation is portrayed with the short exact sequence of Lie groupoids,
   \begin{align}
      \label{eq:gerbe_sequence}
      \begin{tikzcd}[ampersand replacement=\&]
      	1 \& S^1_M \& \G \& {M\times_\pi M} \& 1.
      	\arrow[from=1-1, to=1-2]
      	\arrow[from=1-2, to=1-3]
      	\arrow["\Phi",from=1-3, to=1-4]
      	\arrow[from=1-4, to=1-5]
      \end{tikzcd}
   \end{align}
   An $S^1$-extension is said to be \emph{central}, if for any $g\in \G$ and $z\in S^1$, there holds 
   \[C_g(s(g),z)=(t(g),z),\]
   where $C_g$ is the conjugation by $g$. A central $S^1$-extension of a submersion groupoid as above (or sometimes, its Morita equivalence class) is also known as an \textit{$S^1$-bundle gerbe} over the manifold $N$.
\end{definition} 
At the level of Lie algebroids, the short exact sequence \eqref{eq:gerbe_sequence} reads
\[\begin{tikzcd}
	0 & \R_M & A & {T\F} & 0,
	\arrow[from=1-1, to=1-2]
	\arrow[from=1-2, to=1-3]
	\arrow["{\Phi_*}", from=1-3, to=1-4]
	\arrow[from=1-4, to=1-5]
\end{tikzcd}\]
and centrality implies that the adjoint representations $\Ad\colon \G\curvearrowright \R_M$ and $\ad\colon A\curvearrowright \R_M$ are  trivial. We note that since $\Phi$ is a groupoid morphism, we must have $\Phi=(t,s)$ and so $\Phi_*=\rho$. Moreover, the kernel $\ker\Phi=S^1_M$ is the isotropy group bundle of $\G$, the trivial line bundle $\R_M$ is the isotropy Lie algebra bundle (and thus a bundle of ideals), and $N$ is the leaf space of $\G$. 

Let us now consider a multiplicative Ehresmann connection $\omega\in\A(\G)$ for the groupoid morphism $\Phi=(t,s)$. By \cite{mec}*{Proposition 4.9}, such a connection exists (since $\G$ must be proper) and moreover, the induced linear connection $\nabla$ on $\R_M$, given by $\nabla=\ve(\omega)|_{\Gamma(\frak k)}$, must be the canonical flat connection. This follows from the fact that the differential of $\exp\colon \R_M\ra S^1_M$ maps the induced linear connection, viewed as a distribution  $E^\nabla\subset T(\R_M)$, to the multiplicative Ehresmann connection $\smash{E^{S^1_M}}=T(S^1_M)\cap E$ where $E=\ker\omega$, see \cite{mec}*{equation (2.3)}. This implies the constant unit section of $\R_M$, $x\mapsto (x,1)$, must be flat, hence all constant sections are flat and thus $\nabla$ is indeed the canonical flat connection. Importantly, this holds due to our choice of integration; in general, an IM connection on $A$ for $\R_M$ can give rise to a different flat connection $\nabla$ on $\R_M$.

An important consequence of this observation is that any multiplicative connection $\omega$ on $\G$ for $\R_M$ must be primitive. In fact, an even more general statement holds.
\begin{lemma}
   \label{lem:gerbes_trivial_bss_cohomology}
   On an $S^1$-bundle gerbe $\Phi\colon\G\ra M\times_\pi M$ over a manifold $N$, horizontal multiplicative differential forms of positive degree are cohomologically trivial. In other words,
   \[
   H^1(\Omega^q(\G^{(\bullet)})^\Hor)\coloneq\Omega^q_m(\G)^\Hor / \im(\delta^0_\G) = 0,\quad\text{for all }q> 0.
   \]
   Here, horizontal forms are the forms $\alpha\in\Omega^q(\G)$ satisfying $\iota_X\alpha=0 $ for all $X\in\ker\d \Phi$.
\end{lemma}
\begin{proof}
First note that there holds $\Omega^q_m(\G)^\Hor\cong \Omega^q_m(M\times_\pi M)$, with the isomorphism given by the pullback $\Phi^*$, whose inverse is $\alpha\mapsto \ul\alpha$, 
$
\ul\alpha_{(y,x)}(X_i)\coloneq\alpha_g(\tilde X_i)_i
$
where $\tilde X_i$ are arbitrary lifts of vectors $X_i\in T_{(y,x)}(M\times_\pi M)$ along $\d \Phi$, and $g\colon x\ra y$ is any arrow in $\G$. Well-definedness of this map is a straightforward consequence of multiplicativity and horizontality of $\alpha$. On the submersion groupoid, the Bott--Shulman--Stasheff complex $\Omega^q((M\times_\pi M)^{(\bullet)})$ is acyclic (see \cite{bundle_gerbes_original}*{\sec 8}); in particular, every multiplicative form $\ul\alpha\in\Omega^q_m(M\times_\pi M)$ is cohomologically trivial, i.e., $\ul\alpha=\pr_2^*\beta-\pr_1^*\beta$ for some $\beta\in\Omega^q(M)$, and thus so is every horizontal multiplicative form on $\G$, since $\alpha=\Phi^*\ul\alpha=s^*\beta-t^*\beta$.
\end{proof}
\noindent Looking at the curvature of a MEC $\omega$, the structure equation  (item (ii) after Definition \ref{def:mec_curvature}) reads
\[
\Omega^\omega= \d\omega \in\Omega^2_m(\G).
\]
Since $\Omega^\omega$ is horizontal and multiplicative, the lemma shows there exists a curving $F\in\Omega^2(M)$, i.e.,
\begin{align*}
   \Omega^\omega=\delta^0_\G F=s^*F-t^*F,
\end{align*}
so every MEC for $\Phi\colon\G\ra M\times_\pi M$ is primitive.
Note that due to centrality of extension, invariant forms $\Omega_\ginv^\bullet(M)$ are precisely the ones in the image of the pullback $\pi^*\colon \Omega^\bullet (N)\ra \Omega^\bullet (M)$, 
\[
\Omega_\ginv^\bullet(M) = \pi^*\Omega^\bullet(N).
\]
Hence, for instance, $F$ is only unique up to a form $\beta=\pi^*\tilde\beta$ for some $\tilde\beta\in\Omega^2(N)$, and the curvature 3-form $G=\d F$ must equal $G=\pi^* \tilde G$ for some unique $\tilde G\in\Omega^3(N)$.

Equipped with this knowledge, we are now ready to formulate and study the desired Yang--Mills theory of MECs for the extension \eqref{eq:gerbe_sequence}. As input data, we take the canonical ($\ad$-invariant) metric on $\R_M$, and fix a pseudo-Riemannian metric $g$ on $M$. 
Importantly, note that Lemma \ref{lem:gerbes_trivial_bss_cohomology} implies that any two MECs must differ by a cohomologically trivial 1-form; in other words, $\A(\G)/\im(\delta^0_\G)$ is a singleton. Therefore, the pairs of bundle gerbe connections and their curvings,
\begin{align*}
   \DD(\G)=\set*{(\omega,F)\in\A(\G)\times \Omega^2(M)\given 
   \Omega^\omega=\delta^0_\G F
   },
\end{align*}
form an affine space modelled on the vector space $(\Omega^1(M)\oplus \pi^*\Omega^2(N))/\pi^*\Omega^1(N)$, and the constraint of the variational problem in Theorem \ref{thm:main_theorem_groupoids} is vacuous. Combining everything, we conclude:

\begin{theorem}
   \label{thm:ym_gerbes}
   Let $\G\ra M\times_\pi M$ be an $S^1$-bundle gerbe over a manifold $N$, where $M$ is a compact, oriented, pseudo-Riemannian manifold. A pair $(\omega,F)\in \DD(\G)$ is a critical point for the action 
   \begin{align}
      \label{eq:ym_action_gerbes}
      \S(\omega,F)=\int_M F\wedge \star F + \mu \int_M G\wedge \star G,\quad (\mu\in\R\text{ fixed}),
   \end{align}
   if and only if its curvature pair $(F,G=\d F)$ satisfies the Euler--Lagrange conditions
   \begin{align}
      \d{}^*F&=0,\label{eq:YM.I_gerbes}\\
      F+\mu\d{}^* G&\perp \pi^*\Omega^2(N),\label{eq:YM.II_gerbes}
   \end{align}
   where the covariant coderivative reads $\d{}^*=(-1)^k\star^{-1}\!{\d{}}\,\star $ on forms of degree $k$.
\end{theorem}

Let us first inspect the case $\dim N= 1$, when the situation is as in \sec\ref{sec:vanishing_invariant_forms}. Let us focus on the case when the manifold $M$ is Riemannian. The second equation \eqref{eq:YM.II_gerbes} is vacuously true, and so the focus is on solving \eqref{eq:YM.I_gerbes}. To show it admits a solution, first take any pair $(\omega,F)$. Note that since the curvature 3-form $G=\d F$ vanishes, Hodge theorem ensures $[F]\in H^2_{dR}(M)$ has a (unique) harmonic representative, i.e., there exists a form $\gamma\in\Omega^1(M)$, unique up to a closed 1-form, such that $\d{}^*(F+\d\gamma)=0$. Since $\R_M$ is abelian,  the following pair is thus a critical point of \eqref{eq:ym_action_gerbes}: \[(\omega+\delta^0_\G\gamma, F+\d\gamma)\in\DD(\G).\]

Let us now discuss the case $\dim N\geq 2$. For intuition, we will first show a simple example of a 2-form $F$, such that the pair $(F,G=\d F)$ satisfies equation \eqref{eq:YM.I_gerbes} and relation \eqref{eq:YM.II_gerbes}, for the cases when the leaf space $N$ has dimensions $\dim N=2,3$. Then, we will produce a nontrivial example of a MEC on a bundle gerbe which admits the constructed form $F$ as its curving.
\begin{example}
   \needspace{2cm}
   Let $M=T^4=(\R/2\pi\Z)^4$ be the  standard flat Riemannian 4-torus with coordinates $(x,y,z,t)$, so the metric is $g=\d x^2+\d y^2 + \d z^2+\d t^2$ and the orientation is $\vol=\d x\wedge\d y\wedge \d z\wedge \d t$.\
   \begin{enumerate}
      \item For a leaf space with dimension $\dim N=2$, take $N=T^2$ and $\pi(x,y,z,t)=(x,y)$. In this case, the constant 2-form $F=\d z\wedge \d t$ clearly satisfies the Bianchi identity $\d F=0$, as well as the Euler--Lagrange equation $\d{}^*F=0$. The relation \eqref{eq:YM.II_gerbes}, which now reads 
      \begin{align}
         \label{eq:dimN=2}
         F\perp \pi^*(\Omega^2(N)),
      \end{align}
      is also clearly satisfied since the basic forms are generated by $\d x\wedge \d y$. Note we could also take $F$ to be any other \textit{vertical} harmonic form, that is, one in which all terms contain either $\d z$ or $\d t$, so any $\R$-linear combination in $\langle\d x\wedge \d z, \d x\wedge \d t, \d y\wedge \d z, \d y\wedge \d t, \d z\wedge\d t\rangle$. Since all harmonic forms on $T^4$ are linear combinations of constant forms, these are all the harmonic 2-forms satisfying the partial differential relation \eqref{eq:dimN=2}, so these are all the solutions to the Euler--Lagrange equation and relation.
      \item For a leaf space with dimension $\dim N=3$, take $N=T^3$ and $\pi(x,y,z,t)=(x,y,z)$. Define $F$ as a sum of a coclosed basic part and a  harmonic vertical part; for instance,
      \[
      F=\cos(x) \d y\wedge \d z + \d t\wedge \d x.
      \]
      Now, the 3-form $G=\d F=-\sin(x)\d x\wedge \d y\wedge \d z$ is basic, and importantly, the basic part of $F$ is moreover chosen so that it is an eigenform of the operator $\d{}^*\!\d{}$, with eigenvalue 1. Therefore,
      \[
      F-\d{}^*G=\d t\wedge \d x \perp \pi^*(\Omega^2(N)),
      \]
      hence the pair $(F,G)$ is a solution to \eqref{eq:YM.I_gerbes} and \eqref{eq:YM.II_gerbes} for the constant $\mu=-1$.

      We now show that the constructed 2-form $F$ is indeed an example of a curving of a MEC. To this end, we consider the trivial\footnote{Since $G$ is an exact 3-form on $N$, the only $S^1$-bundle gerbe possibly admitting a MEC with curving $F$ is the trivial one, since this forces the Dixmier--Douady class to vanish.} bundle gerbe over $N=T^3$. In the language of groupoids, a trivial bundle gerbe is an $S^1$-bundle gerbe which is moreover a groupoid of \textit{principal type} \cite{mec}*{\sec 3.5} over the submersion groupoid. That is, $\G$ is the fibred product of the submersion groupoid with a gauge groupoid $\G(P)$ of a principal $S^1$-bundle $p\colon P\ra M$, 
      \begin{align*}
         \qquad\ \G=\G(P)\times_M (M\times_\pi M)=\set{([u_2,u_1],(y,x))\in \G(P)\times (M\times_\pi M)\given p(u_1)=x,p(u_2)=y}.
      \end{align*}
      Let $P=M\times S^1$ be the  trivial principal bundle, so $\G(P)\cong M\times S^1\times M$. The Lie groupoid extension \eqref{eq:gerbe_sequence} is then the following SES of tori over $M=T^4$,
   \begin{align}
      \begin{tikzcd}[ampersand replacement=\&]
      	1 \& T^5 \& T^6 \& T^5 \& 1.
      	\arrow[from=1-1, to=1-2]
      	\arrow[from=1-2, to=1-3]
      	\arrow[from=1-3, to=1-4]
      	\arrow[from=1-4, to=1-5]
      \end{tikzcd}
   \end{align}
   Here, the structure of $T^6$ is the obvious one and so are the maps.
   In accord with the notation above, let us denote the coordinates on $T^6$ by $(x,y,z,t_2,t_1,\varphi)$.  It is not hard to see that 
   \[
   \omega= \d \varphi + (t_1-t_2)\d x\in \Omega^1(\G)
   \]
   defines a multiplicative Ehresmann connection on the bundle gerbe $T^6\ra T^5$. Its abstract curvature 2-form reads  $\Omega^\omega = \d\omega = \d(t_1-t_2)\wedge \d x$, which is precisely $s^*F-t^*F=\delta^0_\G F$. We  conclude $(\omega,F)$ is a critical point of the action functional \eqref{eq:ym_action_G}. 
\end{enumerate}

\end{example}

\begin{remark}
   In the action functional \eqref{eq:ym_action_gerbes} of Theorem \ref{thm:ym_gerbes}, what happens if we omit the 2-curvature term, keeping only the 3-curvature term? In this case, when $\mu\neq 0$, the Euler--Lagrange equation just reads $\d{}^*G\perp \pi^*\Omega^2(N)$. If the metric on  $M$ is chosen $\pi$-invariant (i.e., it descends to a metric on the leaf space $N$), and if  $N$ is also oriented, then $\d{}^*$ preserves pullback forms. Therefore, the Euler--Lagrange equation boils down to 
   \[
   \d{}^* G=0,
   \]
   so the resulting variational problem is about finding pairs $(\omega,F)\in\DD(\G)$ with \textit{harmonic 3-curvature}. At last, for any MEC $\omega$ on $\G$, Hodge's theorem implies there is a curving $F$ with harmonic $G$.
\end{remark}

\subsection{Riemannian manifolds with harmonic curvature}
In this section, we show how the theory of harmonic Riemannian manifolds \cites{harmonic_riemannian_curvature, harmonic_riemannian_dim_4, jost_riemannian_geometry} can  be viewed in the developed framework. Roughly, we will see it corresponds to Yang--Mills theory for the general linear algebroid. Since it is a transitive and integrable example, it has a classical interpretation: it corresponds to Yang--Mills theory for the (orthogonal) frame bundle; the purpose here is to expose the simplicity of classical Yang--Mills theory when expressed with Lie algebroids.

Suppose $V\ra M$ is a vector bundle with a Riemannian fibrewise metric $\kappa=\inner\cdot\cdot_V$, and consider the general linear algebroid $\frak{gl}(V)$. Let $\frak o(V)$ denote its subalgebroid consisting of all derivations compatible with the metric $\kappa$, as defined in \eqref{eq:metric_derivation}. 
The associated short exact Atiyah  sequence reads
\[\begin{tikzcd}
	0 & {\End_\kappa(V)} & {\frak{o}(V)} & TM & 0,
	\arrow[from=1-1, to=1-2]
	\arrow[from=1-2, to=1-3]
	\arrow[from=1-3, to=1-4]
	\arrow[from=1-4, to=1-5]
\end{tikzcd}\]
where $\End_\kappa(V)$ denotes the Lie algebra bundle of endomorphisms of the fibres of $V$ which are skew-symmetric with respect to the metric $\kappa$. A splitting of this sequence is just a metric linear connection on $V$, and the space of metric connections is an affine space over the space $\Omega^1(M;\End_\kappa (V))$ of skew-symmetric endomorphism-valued 1-forms on $M$. For any splitting $\sigma\colon TM\ra \frak{o}(V)$, let us accordingly denote $\nabla_X=\sigma(X)$ for any $X\in TM$. The induced linear connection $\nabla^\sigma$ on $\End_\kappa(V)$ is then simply the restriction of the induced connection $\nabla^{\End V}$ on $\End(V)$, which we will denote by the same symbol, $\nabla^\sigma=\nabla^{\End V}$. The curvature of a splitting $\sigma$ is just the curvature tensor of $\nabla$:
\[
	F^\sigma\in\Omega^2(M;\End V),\quad F^\sigma(X,Y)\xi=\nabla_{[X,Y]}\xi-[\nabla_X,\nabla_Y]\xi=-R^\nabla(X,Y)\xi,
\]
which also has values in skew-symmetric endomorphisms of $V$. Now, the metric $\kappa$ induces a metric $\inner\cdot\cdot_{\End V}$ on $\End V$, whose restriction to the skew-symmetric endomorphisms $\End_{\kappa}(V)$ is ad-invariant. 
To construct a Yang--Mills theory, we need to also fix a metric and an orientation on the base $M$.
The action functional then reads
\[
\S(\nabla)=\int_M\innersmall{R^\nabla}{R^\nabla}_{\End V}\vol_M.
\]
Given any metric connection $\nabla$, the induced connection $\nabla^{\End V}$ is compatible with $\inner\cdot\cdot_{\End V}$, hence $\nabla$ is critical if and only if its curvature satisfies
\[
\d{}^{\nabla^{\End V}}\!\star R^\nabla=0.
\]
Since the Bianchi identity $\d{}^{\nabla^{\End V}}R^\nabla=0$ always holds, this is equivalent to saying that $R^\nabla$ is harmonic. If we take $V=TM$, this amounts to the vanishing divergence of $R^\nabla$, and examples of such Riemannian manifolds include 4-dimensional Einstein manifolds and conformally flat 4-manifolds of constant scalar curvature. We refer the reader to \cites{harmonic_riemannian_curvature, harmonic_riemannian_dim_4} for more details and interesting properties of Riemannian manifolds with harmonic curvature.

\section{Related generalizations}
This section is devoted to a brief and somewhat speculative discussion about the relationship of our generalization of Yang--Mills theory with other generalizations which have already been developed \cites{higher_ym, fischer}. We suspect that the different generalizations mentioned below, as well as ours, are particular cases within the umbrella framework of so-called \textit{$\PB$-groupoids}; these are Lie groupoid objects in the category of principal bundles. A general notion of connections, as well as a Yang--Mills theory, are yet to be developed for these geometric objects.

\SkipTocEntry\subsection*{Yang--Mills theory on principal 2-bundles}
  In \cite{higher_ym}, Baez develops a Yang--Mills theory for (trivial) principal 2-bundles. Roughly, these are higher-geometric analogues of principal bundles, where the structure Lie group is replaced with a structure Lie 2-group \cite{higher_gauge}. In this setting, connections provide a means for parallel transport over surfaces instead of only paths; from a physics perspective, this is interpreted as parallel transport of strings along branes. The Euler--Lagrange equations found in \cite{higher_ym}*{Theorem 21} are strikingly similar to our equations \eqref{eq:YM.I} and \eqref{eq:YM.IIe}. The similarity is already apparent at the level of action functionals, if one compares \cite{higher_ym}*{equation (9)} with our action  \eqref{eq:ym_action}. 
   Despite the similarities, the framework of principal 2-bundles is conceptually very different from ours, and there is no clear way to view one theory as an instance of the other. Instead, as mentioned, we believe the similarities come from the fact that the two (global) frameworks are particular cases within a more general theory on \textit{$\PB$-groupoids}.

\SkipTocEntry\subsection*{Yang--Mills theory on principal LGB-bundles}
In \cites{fischer_massless, fischer}, Fischer develops a Yang--Mills theory for principal LGB-bundles. Roughly, a principal LGB-bundle is a generalization of a principal bundle, where the structure Lie group is replaced with a Lie group bundle (LGB). Infinitesimally, this means the structure constants of the Lie algebra become functions. The notion of a connection in \textit{loc.\,cit.}\ depends on a fixed choice of a Cartan connection on the structure Lie group bundle, and the dynamical fields are the connections on the principal bundle, which are compatible with the fixed Cartan connection. What the author there calls a field redefinition, appears  related to our Lemma \ref{lemma:affine_primitive}. 
However, as with principal 2-bundles, the framework of principal LGB-bundles is conceptually entirely different from ours, and there is no clear overlap between the two theories.

\vspace{0.5em}

It appears that the similarities and differences between the various settings are a manifestation of the fact that they are merely particular cases of a more general framework of $\PB$-groupoids. Roughly, a $\PB$-groupoid is a surjective submersion of Lie groupoids, with a principal (right) action of a Lie 2-groupoid $H_2\rra H_1\rra H_0$. The situation is depicted with the following diagram, where the curved arrows are the moment maps of the actions by the groupoids $H_2$ and $H_1$. 
\begin{align}
   \label{eq:pb}
   \begin{tikzcd}[ampersand replacement=\&,row sep=tiny]
   	\& H_2 \&[-2.4em] \curvearrowright \&[-2.4em] {P_G} \& G \\
      H_0 \& \& \& \& \\
   	\& H_1 \& \curvearrowright \& {P_M} \& M
   	\arrow[from=1-4, to=1-5]
   	\arrow[shift left, from=1-4, to=3-4]
   	\arrow[shift right, from=1-4, to=3-4]
   	\arrow[shift left, from=1-5, to=3-5]
   	\arrow[shift right, from=1-5, to=3-5]
   	\arrow[shift left, from=1-2, to=3-2]
   	\arrow[shift right, from=1-2, to=3-2]
   	\arrow[shift left, from=1-2, to=2-1]
   	\arrow[shift right, from=1-2, to=2-1]
   	\arrow[shift left, from=3-2, to=2-1]
   	\arrow[shift right, from=3-2, to=2-1]
   	\arrow[from=3-4, to=3-5]
   	\arrow[from=1-4, to=2-1, bend left=-45]
   	\arrow[from=3-4, to=2-1, bend left=45]
   \end{tikzcd}
\end{align}
It is already known that $\PB$-groupoids simultaneously generalize ordinary principal bundles, principal 2-bundles, principal groupoid-bundles (in particular, principal LGB-bundles), and Lie groupoid extensions. They also provide a natural setting for constructing the frame bundle of a $\vb$-groupoid \cites{pb_groupoids, pb_vb_groupoids}. We speculate that the various notions of connections in the mentioned settings above are just particular instances of connections on particular cases of $\PB$-groupoids. 

\appendix

\section{Ad-invariance for Lie groupoids versus Lie algebroids}
Let $\G$ be a Lie groupoid integrating a Lie algebroid $A$. This appendix is dedicated to showing that if $\G$ is $s$-connected,  $\ad$-invariance of a metric $\inner\cdot\cdot_{\frak k}$ on a bundle of ideals $\frak k$ of $A$ is equivalent to its $\Ad$-invariance. Preliminarily, we must note that any bundle of ideals on a Lie algebroid $A$ is necessarily also a bundle of ideals on $\G$, provided $\G$ has connected source fibres, so we automatically get a representation $\Ad\colon \G\curvearrowright\frak k$ from $\ad\colon A\curvearrowright \frak k$. This fact is shown in \cite{rigidity_poisson_submanifolds}*{Appendix B}, and our desired result relies on the method of proof therein.

\begin{lemma}
	\label{lem:pairing_A_invariant_implies_G_invariant}
	Let $\inner\cdot\cdot_{\frak k}$ be an $\ad$-invariant metric on a bundle of ideals $\frak k$ of a Lie algebroid $A$, and suppose a Lie groupoid $\G$ integrates $A$. If $\G$ has connected $s$-fibres, then $\inner{\cdot}{\cdot}_{\frak k}$ is also $\Ad$-invariant, that is, for any $g\in \G$ and $\xi,\eta\in\frak k_{s(g)}$, there holds
	\begin{align}
		\label{eq:Ad_invariance}
		\inner{\Ad_g\xi}{\Ad_g\eta}_{\frak k}=\inner{\xi}{\eta}_{\frak k}.
	\end{align}
	Conversely, $\Ad$-invariance implies $\ad$-invariance regardless of connectivity of $s$-fibres of $\G$.
\end{lemma}
\begin{proof}
	This is shown by a standard trick, namely, we first show that the equality above holds for all $g$ in a neighborhood of the units $u(M)\subset \G$. Since any neighborhood of the units generates $\G$ by $s$-connectedness (for instance, see \cite{mackenzie}*{Proposition 1.5.8}), the equality \eqref{eq:Ad_invariance} then holds for all $g\in \G$. Hence, it is enough to show that for any $\xi,\eta\in\Gamma(\frak k)$, there holds
	\begin{align}
		\label{eq:ad_implies_Ad_intermediate}
		\inner{\Ad_{g_\lambda}(\xi_{s(g_\lambda)})}{\Ad_{g_\lambda}(\eta_{s(g_\lambda)})}_{\frak k}=\inner{\xi_{s(g_\lambda)}}{\eta_{s(g_\lambda)}}_{\frak k},\ \text{where}\ g_\lambda\coloneqq\phi^{\alpha^L}_\lambda(1_x),
	\end{align}
	 for any  $\alpha\in\Gamma(A)$ and all times $\lambda$ for which the integral path of $\alpha^L$ through $1_x$ is defined. 

	 To show this, we utilize the notion of \textit{derivations}, also known as \textit{covariant differential operators} on a vector bundle $V\ra M$, which are the sections of the general linear algebroid $\frak{gl}(V)\Ra M$,
    \[
    \Gamma(\frak{gl}(V))=\set*{\big(D\colon\Gamma(V)\xrightarrow{\R\text{-lin.}} \Gamma(V), X\in\vf(M)\big)\given  D(f\sigma)=fD(\sigma)+X(f)\sigma}.
    \] 
    We recall that the \emph{flow} of a derivation $(D,X)\in\Gamma(\frak{gl}(V))$ is defined as the vector bundle map $\Phi^D_\lambda\colon V\ra V$ covering the flow $\phi^{X}_\lambda$ of $X\in\vf(M)$, denoted $\Phi^D_\lambda(x)\colon V_x\ra V_{\phi^X_\lambda(x)}$, which is the solution to the following differential equation.
	 \begin{equation}
		\label{eq:flow_of_derivation}
		\begin{aligned}
			\frac{d}{d\lambda}(\Phi^D_\lambda)^*(\xi)&=(\Phi^D_\lambda)^*(D\xi),\\ 
			\Phi^D_0&=\id_V.
		\end{aligned}
		\qquad
		\vcenter{\hbox{
			\begin{tikzcd}
				V & V \\
				M & M
				\arrow["{\Phi^{D}_\lambda}", from=1-1, to=1-2]
				\arrow[from=1-1, to=2-1]
				\arrow[from=1-2, to=2-2]
				\arrow["{\phi^{X}_\lambda}"', from=2-1, to=2-2]
			\end{tikzcd}
		}}
	\end{equation}
	Here, we have denoted $(\Phi^D_\lambda)^*(\xi)_x=\Phi^D_{-\lambda}(\phi^X_\lambda(x))\xi_{\phi^X_\lambda(x)}$. Given a metric $\inner\cdot\cdot_V$ on $V$, a derivation $(D,X)$ is said to be \emph{compatible} with the metric, if there holds
\begin{align}
	\label{eq:metric_derivation}
		X\inner{\xi}{\eta}_V=\inner{D\xi}{\eta}_V+\inner{\xi}{D\eta}_V\, \text{ for all }\,\xi,\eta\in\Gamma(V).
\end{align}
	Now, this is equivalent to saying that the flow $\Phi^D_\lambda$ acts on $V$ by isometries --- proving this amounts to proving the standard fact that parallel transport of a metric connection acts on $V$ by isometries.
	In our case, we take the vector bundle $V=\frak k$. Any section  $\alpha\in\Gamma(A)$ defines a derivation 
	$([\alpha,\cdot],\rho(\alpha))\in\Gamma(\frak{gl}(\frak k))$, and we may assume for simplicity that $\rho(\alpha)\in\vf(M)$ is complete, so that $\alpha^L$ is also complete by \cite{mackenzie}*{Theorem 3.6.4}. It is shown in \cite{rigidity_poisson_submanifolds}*{Equation (43)} that its flow equals 
	\[
	\Phi_\lambda^{[\alpha,\cdot]}=\Ad_{\exp(\lambda\alpha)},
	\]
	that is, $\smash{\Phi^{[\alpha,\cdot]}_\lambda}$ equals the induced Lie algebroid map associated to the inner automorphism $I_{\exp(\lambda\alpha)}$ of $\G$, induced by the bisection $\exp(\lambda\alpha)$; see the discussion after the proof for notation. Using ad-invariance of $\inner\cdot\cdot_{\frak k}$, we conclude $\Ad_{\exp(\lambda\alpha)}$ acts on $\frak k$ by isometries. Since for any $\xi\in\frak k_x$, there holds $\Ad_{\exp(\lambda\alpha)}(\xi)=\Ad_{g_\lambda^{-1}}(\xi)$ where $g_\lambda$ is as in equation \eqref{eq:ad_implies_Ad_intermediate}, we are done.
\end{proof}
In the proof above, a \textit{bisection} $b\in\mathrm{Bis}(\G)$ of the groupoid $\G$ is a map
\begin{align}
\label{eq:bisection_gauge}
	b\colon M\ra \G,\quad t\circ b=\id_M,\quad s\circ b=\varphi,
\end{align}
where $\varphi\colon M\ra M$ is a diffeomorphism. It is called a \textit{static} bisection, if $\varphi=\id_M$, in which case $b$ maps into the isotropy of $\G$. The associated inner automorphism is defined as
\begin{equation}
\label{eq:inner_automorphism}
\begin{aligned}
	&I_b\colon \G\ra \G,\\ 
	&I_b(g)=b(t(g))^{-1}g b(s(g)).
\end{aligned}
\qquad
\vcenter{\hbox{
	\begin{tikzcd}
		\G & \G \\
		M & M
		\arrow["{I_b}", from=1-1, to=1-2]
		\arrow[shift left, from=1-1, to=2-1]
		\arrow[shift right, from=1-1, to=2-1]
		\arrow[shift left, from=1-2, to=2-2]
		\arrow[shift right, from=1-2, to=2-2]
		\arrow["\varphi", from=2-1, to=2-2]
	\end{tikzcd}
}}
\end{equation}
As portrayed in the diagram, $I_b$ is a Lie groupoid automorphism covering $\varphi$ on the base. The map $\Ad_b\coloneq(I_b)_*\colon A\ra A$ denotes the induced Lie algebroid morphism over $\varphi$, and $\exp(\lambda\alpha)\colon M\ra \G$ denotes the bisection $\exp(\lambda\alpha)(x)=\phi^{\smash{\alpha^L}}_\lambda(1_x)$.

\Addresses

\end{document}